\documentclass[a4paper,11pt,reqno]{amsart}
\usepackage{amsmath,amssymb}
\usepackage{mathrsfs}
\usepackage{amsfonts}
\usepackage[dvipsnames]{xcolor}
\usepackage{graphicx}
\usepackage{float}  
\usepackage{cite}
\usepackage{cases}
\usepackage{comment}
\usepackage{indentfirst}
\allowdisplaybreaks

\theoremstyle{plain}
\newtheorem{theorem}{Theorem}[section]
\newtheorem{lemma}[theorem]{Lemma}
\newtheorem{proposition}[theorem]{Proposition}

\theoremstyle{definition}

\theoremstyle{remark}

\numberwithin{equation}{section}

\usepackage[textwidth=16.0cm,textheight=21.0cm,centering]{geometry}
\usepackage[colorlinks=true,citecolor=blue,linkcolor=blue]{hyperref}

\begin{document}
\title[On the Dirichlet problem for the degenerate $2$-Hessian equation]
{On the Dirichlet problem for the degenerate $2$-Hessian equation}

\author[Y. Lyu]{Yasheng Lyu}
\address{School of Mathematics and Statistics, Xi'an Jiaotong University, Xi'an, Shaanxi 710049, People's Republic of China}
\email{lvysh21@stu.xjtu.edu.cn}

\date{}

\begin{abstract}
In this paper, we establish global $C^{1,1}$ solvability of the Dirichlet problem for the degenerate $2$-Hessian equation on bounded strictly mean convex $C^{3,1}$ domains $\Omega$, with general $C^{3,1}$ boundary values and nonnegative right-hand sides $f\in C^{1,1}(\overline{\Omega})$, resolving the $k=2$ case of a longstanding open problem.
We uncover a global semiconvexity structure for general $2$-admissible solutions by proving a lower bound for $\sigma_{3}[D^{2}u]$ independent of $\inf_{\Omega}f$.
For $3\leq k\leq n-1$, we also establish global semiconvexity for $k$-admissible solutions on the unit ball with $C^{3}$ boundary values and $f^{1/(k-1)}\in C^{1,1}(\overline{B}_{1})$, with a bound independent of $\inf_{B_{1}}f$.
A counterexample shows that the $C^{3,1}$ boundary value assumption for global $C^{1,1}$ solvability is sharp in the H\"older scale.
\end{abstract}

\maketitle

\section{Introduction}

In this paper, we consider the Dirichlet problem for the degenerate $k$-Hessian equation with $k=2$
\begin{equation}\label{eqn1.1}
\begin{cases}
\sigma_{k}\big[D^{2}u\big]=f\geq0\quad\text{in}\ \Omega,\\
u=\varphi\quad\text{on}\ \partial\Omega,
\end{cases}
\end{equation}
where $\Omega$ is a bounded domain in $\mathbb{R}^{n}$, $n\geq2$, and $D^{2}u$ denotes the Hessian matrix of $u$.
For an $n\times n$ real symmetric matrix $A$, we denote by $\sigma_{k}[A]$ the $k$-th elementary symmetric polynomial of its eigenvalues $\lambda(A)=(\lambda_{1},\dots,\lambda_{n})$; that is,
\[
\sigma_{k}[A]=\sigma_{k}(\lambda(A))
:=\sum_{1\leq i_{1}<\cdots<i_{k}\leq n}\lambda_{i_{1}}\cdots\lambda_{i_{k}},\quad k=1,2,\dots,n.
\]
Equivalently, $\sigma_{k}[A]$ is the sum of all $k\times k$ principal minors of $A$.

Following Caffarelli-Nirenberg-Spruck \cite{CNS}, a function $u\in C^{2}(\Omega)\cap C^{0}(\overline{\Omega})$ is called $k$-admissible if $\lambda(D^{2}u(x))\in\Gamma_{k}$ for every $x\in\Omega$, where
\[
\Gamma_{k}:=\{\lambda\in\mathbb{R}^{n}:\ \sigma_{j}(\lambda)>0\ \text{for all}\ 1\leq j\leq k\}.
\]
We also use $\Gamma_{k}$ for the corresponding cone of symmetric matrices.
The elementary symmetric functions satisfy
\[
\frac{\partial\sigma_{k}(\lambda)}{\partial\lambda_{i}}
=\sigma_{k-1;i}(\lambda)>0\quad\text{in }\Gamma_{k},\quad 1\leq i\leq n,
\]
where the index after the semicolon denotes deletion of that component.
Moreover, $\sigma_{k}^{1/k}$ is concave on $\Gamma_{k}$, and the matrix operator $\sigma_{k}^{1/k}[A]$ is elliptic and concave there; see \cite{CNS}.
For degenerate equations, admissibility is understood in the closed cone $\overline{\Gamma}_{k}$.
For $C^{1,1}$ solutions, both the equation and admissibility are understood almost everywhere.

For $1\leq k\leq n-1$, a hypersurface is $k$-convex if its principal curvatures $\kappa=(\kappa_{1},\dots,\kappa_{n-1})$ satisfy $\sigma_{j}(\kappa)\geq0$ for $1\leq j\leq k$, and strictly $k$-convex if these quantities have a uniform positive lower bound.
For boundaries of domains, we use the curvature convention of \cite{CNS}, with positive principal curvatures on the unit sphere.
In particular, $1$-convexity is mean convexity.
Caffarelli-Nirenberg-Spruck \cite{CNS} proved that smooth positive right-hand sides and smooth boundary values give a unique smooth admissible solution on smooth strictly $(k-1)$-convex domains.
For constant boundary values and $f>0$ on $\overline{\Omega}$, strict $(k-1)$-convexity is also necessary for a $C^{2}(\overline{\Omega})$ admissible solution.
Subsolution and Hessian quotient extensions were developed by Guan \cite{Guan1994} and Trudinger \cite{Trudinger}; for Hessian measures, variational methods, and an overview of the theory, see \cite{TWII,CW,WangHessian}.
The equation is nondegenerate when $\inf_{\Omega} f>0$.
Here we allow $f$ to vanish on $\overline{\Omega}$ and seek estimates independent of a positive lower bound for $f$.

For the Monge-Amp\`ere equation, $k=n$, after the early work of Trudinger-Urbas \cite{TU1983,TU1984}, Caffarelli-Nirenberg-Spruck \cite{CNSDegenerate} proved global $C^{1,1}$ solvability for the homogeneous equation on uniformly convex $C^{3,1}$ domains with $C^{3,1}$ boundary values.
Boundary estimates and subsolution approaches were further developed in \cite{Hong,Guan1997,GuanBo}.
Guan-Trudinger-Wang \cite{GTW} proved global $C^{1,1}$ solvability on uniformly convex $C^{3,1}$ domains with general $C^{3,1}$ boundary values under
\[
f^{1/(n-1)}\in C^{1,1}(\overline{\Omega}),\quad f\geq0.
\]
Their estimate is independent of $\inf_{\Omega} f$, and the exponent $1/(n-1)$ is optimal by Wang's counterexamples \cite{Wang1995}.

For general Hessian equations, Trudinger \cite{TrudingerBalls} treated degenerate equations in balls, and Krylov developed probabilistic and analytic approaches to global derivative estimates for degenerate fully nonlinear equations \cite{KrylovValue,KrylovWeak,KrylovBarriers,KrylovGeneral,KrylovBellman}.
Ivochkina-Trudinger-Wang \cite{ITW} gave a PDE proof of global second derivative estimates independent of the positive lower bound for the right-hand side, yielding $C^{1,1}$ solvability for degenerate $k$-Hessian equations under $f^{1/k}\in C^{1,1}(\overline{\Omega})$ on strictly $(k-1)$-convex domains.
They proposed extending this result to
\begin{equation}\label{eqn1.optimal}
f^{1/(k-1)}\in C^{1,1}(\overline{\Omega}),\quad 2\leq k\leq n-1;
\end{equation}
see \cite[Section 4.2]{ITW}.
The optimality of this exponent was pointed out by them with no proof, and was proved in detail by Dinew-Pli\'s-Zhang \cite[Proposition 1.3]{DPZ}.
Under \eqref{eqn1.optimal}, Dong \cite{Dong} obtained global estimates for homogeneous boundary values.
Jiao-Wang \cite{JW} established global $C^{1,1}$ estimates for convex solutions with general $C^{3,1}$ boundary values on uniformly convex $C^{3,1}$ domains.
In \cite{LyuD}, the author obtained a weak interior estimate for a positive right-hand side satisfying the optimal condition on a neighborhood of $\overline{\Omega}$, and a boundary estimate under an additional lower bound for $\Delta u$ near boundary points where $\Delta u$ is large.
Global $C^{1,1}$ solvability under \eqref{eqn1.optimal} remained open for general admissible solutions with general boundary values.

We resolve this problem for $k=2$ on strictly mean convex domains.
The main new ingredient is a global lower bound for $\sigma_{3}[D^{2}u]$, which yields semiconvexity uniformly as the right-hand side vanishes.
In the following estimates, dependence on $\Omega$ includes its diameter, the $C^{3,1}$ bounds in a fixed finite boundary atlas, a fixed boundary neighborhood, and a positive lower bound for the mean curvature.

\begin{theorem}\label{thm1.1}
Let $\partial\Omega\in C^{3,1}$ be strictly mean convex, and let $\varphi\in C^{3,1}(\partial\Omega)$ and $f\in C^{1,1}(\overline{\Omega})$ with $f>0$ on $\overline{\Omega}$.
Let $u\in C^{3,1}(\overline{\Omega})$ be a $2$-admissible solution of \eqref{eqn1.1} with $k=2$.
Then
\[
\sigma_{3}[D^{2}u]\geq-C\quad\text{on }\overline{\Omega},
\]
which implies $D^{2}u\geq-CI$, where $C$ depends only on $n$, $\Omega$, $\|\varphi\|_{C^{3,1}(\partial\Omega)}$, and $\|f\|_{C^{1,1}(\overline{\Omega})}$, but is independent of $\inf_{\Omega} f$.
\end{theorem}

The semiconvexity estimate supplies the spectral control needed for the boundary double normal derivative estimate and leads to the following global solvability result.

\begin{theorem}\label{thm1.2}
Let $\partial\Omega\in C^{3,1}$ be strictly mean convex, and let $\varphi\in C^{3,1}(\partial\Omega)$ and $f\in C^{1,1}(\overline{\Omega})$ with $f\geq0$ on $\overline{\Omega}$.
Then \eqref{eqn1.1} with $k=2$ admits a unique $2$-admissible solution $u\in C^{1,1}(\overline{\Omega})$.
\end{theorem}

Theorem \ref{thm1.2} resolves the $k=2$ case of the open problem in \cite[Section 4.2]{ITW}.
Moreover, the boundary value assumption is sharp in the H\"older scale: for every $\alpha\in(0,1)$, a smooth strictly convex regularization of \cite[Example 1]{CNSDegenerate} has uniformly bounded $C^{3,\alpha}$ boundary values and $C^{1,1}$ right-hand sides, but unbounded Hessians.

We also obtain a semiconvexity estimate for other $k$-Hessian equations on the unit ball, by adapting Trudinger's tangential trace estimate \cite{TrudingerBalls} to the optimal condition on the right-hand side.
For $k=n$, admissibility already implies convexity.

\begin{theorem}\label{thm1.3}
Let $n\geq4$ and $3\leq k\leq n-1$, let $\varphi\in C^{3}(\partial B_{1})$, and let $f>0$ in $B_{1}$ with $f^{1/(k-1)}\in C^{1,1}(\overline{B}_{1})$.
Let $u\in C^{3,1}(\overline{B}_{1})$ be a $k$-admissible solution of \eqref{eqn1.1} with $\Omega=B_{1}$.
Then
\[
D^{2}u\geq-CI\quad\text{on }\overline{B}_{1},
\]
where $C$ depends only on $n$, $k$, $\|\varphi\|_{C^{3}(\partial B_{1})}$, and $\|f^{1/(k-1)}\|_{C^{1,1}(\overline{B}_{1})}$, but is independent of $\inf_{B_{1}}f$.
\end{theorem}

The rest of the paper is organized as follows.
Section 2 proves Theorem \ref{thm1.1}.
Section 3 establishes the global estimate and Theorem \ref{thm1.2}.
Section 4 gives the boundary counterexample.
Section 5 proves Theorem \ref{thm1.3}.

\section{Global semiconvexity of \texorpdfstring{$2$}{2}-admissible solutions}

Throughout this section, $k=2$.
To prove Theorem \ref{thm1.1}, we first bound $\sigma_{3}[D^{2}u]$ from below on the boundary and then throughout the domain.
All constants have the dependencies stated in that theorem and are independent of $\inf_{\Omega}f$ and of second and higher derivatives of $u$.
Since $f>0$ on $\overline{\Omega}$, continuity gives $D^{2}u\in\Gamma_{2}$ there.

The matrix derivatives of the Hessian operators are
\begin{align*}
\sigma_{k}^{ij}[A]&:=\frac{\partial\sigma_{k}[A]}{\partial A_{ij}},\quad
\sigma_{k}^{ij,st}[A]:=\frac{\partial^{2}\sigma_{k}[A]}{\partial A_{ij}\partial A_{st}}.
\end{align*}
Unless an argument is displayed, they are evaluated at $D^{2}u$.
Repeated indices are summed in contractions; identities for matrix entries are componentwise.
Matrix norms are Frobenius norms.

\subsection{The boundary lower bound}

It is clear that 
\begin{equation}\label{eqn2.elementary}
\sigma_{2}^{ij}=(\Delta u)\delta_{ij}-u_{ij},\quad |D^{2}u|^{2}<(\Delta u)^{2},\quad \sum_{i=1}^{n}\sigma_{2}^{ii}=(n-1)\Delta u.
\end{equation}
On strictly mean convex domains, the estimates in \cite[Proposition 2.1, (2.2), and (4.2)--(4.3)]{JW} give
\begin{equation}\label{eqn2.first}
\|u\|_{C^{1}(\overline{\Omega})}\leq C.
\end{equation}
In an orthonormal frame with last vector the inward unit normal $\nu$, we write $\sigma_{2}^{\nu\nu}=\sigma_{2}^{ij}\nu_{i}\nu_{j}$.
The same estimates give
\begin{equation}\label{eqn2.boundaryblocks}
|u_{\alpha\beta}|+|u_{\alpha\nu}|\leq C\ \ (\alpha,\beta<n),\quad
\sum_{\alpha=1}^{n-1}u_{\alpha\alpha}=\sigma_{2}^{\nu\nu}>0.
\end{equation}
For $k=2$, their condition on the right-hand side is $|Df|\leq C$, so $f\in C^{1,1}(\overline{\Omega})$ suffices without an exterior extension.
The boundary equation is
\begin{align*}
f&=u_{\nu\nu}\sum_{\alpha=1}^{n-1}u_{\alpha\alpha}
+\sigma_{2}[(u_{\alpha\beta})]-\sum_{\alpha=1}^{n-1}u_{\alpha\nu}^{2}.
\end{align*}
Greek indices in boundary formulas range over the tangential directions.

At a boundary point, choose graph coordinates with
\begin{align*}
\partial\Omega&=\{x_{n}=\rho(x')\},\quad \Omega=\{x_{n}>\rho(x')\},\quad D'\rho(0)=0,
\end{align*}
and use the local barrier from \cite[Section 3]{LyuD},
\begin{align*}
V&=(x_{n}-\rho)^{2}-\beta_{1}(x_{n}-\rho)-\beta_{2}|x'|^{2},\quad
\Omega_{r}=\{|x'|<r,\ 0<x_{n}-\rho(x')<r^{2}\}.
\end{align*}
Strict mean convexity permits fixed small $\beta_{1},\beta_{2},r>0$ such that
\begin{equation}\label{eqn2.barrierinequality}
D^{2}V-cI\in\Gamma_{2},\quad \sigma_{2}^{ij}V_{ij}\geq c\Delta u,\quad
V\leq-c(x_{n}-\rho)-\beta_{2}|x'|^{2},\quad V_{n}\leq-c.
\end{equation}
The inequalities hold throughout $\Omega_{r}$, uniformly in a finite boundary atlas.
The inequality for $\sigma_{2}^{ij}V_{ij}$ follows from \eqref{eqn2.elementary} and $\operatorname{tr}A\geq|A|$ for $A\in\overline{\Gamma}_{2}$.

We first prove an estimate at a boundary contact point.
Suppose that $u_{n}(x',\rho(x'))$ admits a $C^{1,1}$ upper support $g$ on a fixed graph neighborhood; that is,
\[
u_{n}(x',\rho(x'))\leq g(x'),\quad u_{n}(0)=g(0).
\]
The Taylor remainder of $g$ and \eqref{eqn2.first} give $K=C(1+\|g\|_{C^{1,1}})$ such that
\begin{align*}
u_{n}-g(0)-D'g(0)\cdot x'+KV&\leq0\quad\text{on }\partial\Omega_{r}.
\end{align*}
At a positive interior maximum, the normal derivative condition, \eqref{eqn2.elementary}, and \eqref{eqn2.barrierinequality} give
\begin{align*}
\Delta u&\geq|u_{nn}|=-KV_{n}\geq cK,\\
0&\geq f_{n}+K\sigma_{2}^{ij}V_{ij}\geq-\|Df\|_{\infty}+cK\Delta u>0,
\end{align*}
provided the fixed constant in $K$ is sufficiently large.
Thus the comparison holds in $\Omega_{r}$, and inward differentiation at the origin yields
\begin{equation}\label{eqn2.supportcontact}
u_{nn}(0)\leq C(1+\|g\|_{C^{1,1}}).
\end{equation}
Here $C$ also depends on the size of the fixed neighborhood.

We use \eqref{eqn2.supportcontact} to bound
\[
Z=\max_{\partial\Omega}
\frac{f-\sigma_{2}[(u_{\alpha\beta})]}{\sigma_{2}^{\nu\nu}}.
\]
It suffices to consider $Z$ above a fixed positive threshold.
Center the graph chart at a point where this maximum is attained, and write
$\widetilde\varphi(x')=\varphi(x',\rho(x'))$ and
$G=I_{n-1}+D'\rho(D'\rho)^{T}$.
Differentiating the boundary condition gives, in the orthonormal tangent frame induced by the graph,
\[
(u_{\alpha\beta})
=G^{-1/2}(D'^{2}\widetilde\varphi-u_{n}D'^{2}\rho)G^{-1/2}.
\]
Here $u_{n}$ and $f$ are evaluated at $(x',\rho(x'))$.

Strict mean convexity gives $\operatorname{tr}(G^{-1}D'^{2}\rho)\geq h_{0}$.
Define $g$ implicitly by
\begin{align*}
\operatorname{tr}\bigl[G^{-1}(D'^{2}\widetilde\varphi-gD'^{2}\rho)\bigr]
+\frac{\sigma_{2}\bigl[G^{-1}(D'^{2}\widetilde\varphi-gD'^{2}\rho)\bigr]-f}{Z}&=0,
\quad g(0)=u_{n}(0).
\end{align*}
The matrices in this formula are similar to symmetric matrices, and $\sigma_{2}$ is evaluated on their eigenvalues.
Replacing $g$ by $s$, the left-hand side has uniformly bounded $C^{1,1}$ norm in $(x',s)$ and $s$-derivative at most $-h_{0}/2$ for $|s|\leq\|Du\|_{\infty}+1$ and large $Z$.
The implicit function theorem gives $g$ on a neighborhood of uniform size, with $\|g\|_{C^{1,1}}\leq C$.
The left-hand side is nonnegative at $s=u_{n}$ by the definition of $Z$, so monotonicity gives $u_{n}\leq g$.
At this contact point, \eqref{eqn2.supportcontact} and the boundary equation give $Z\leq u_{nn}(0)\leq C$.
Consequently, the boundary equation and \eqref{eqn2.boundaryblocks} give
\begin{align*}
\sigma_{2}[(u_{\alpha\beta})]\geq f-C\sum_{\alpha=1}^{n-1}u_{\alpha\alpha},\quad
u_{\nu\nu}\sum_{\alpha=1}^{n-1}u_{\alpha\alpha}\leq C.
\end{align*}
Expanding the third-order principal minors and using the bounded tangential and mixed derivatives gives
\begin{align*}
\sigma_{3}[D^{2}u]&\geq u_{\nu\nu}\sigma_{2}[(u_{\alpha\beta})]-C.
\end{align*}
If $u_{\nu\nu}\geq0$, then
\begin{align*}
u_{\nu\nu}\sigma_{2}[(u_{\alpha\beta})]
&\geq-Cu_{\nu\nu}\sum_{\alpha=1}^{n-1}u_{\alpha\alpha}\geq-C.
\end{align*}
If $u_{\nu\nu}<0$, it is bounded because $\Delta u>0$ implies $u_{\nu\nu}>-\sum_{\alpha=1}^{n-1}u_{\alpha\alpha}\geq-C$.
Thus
\begin{equation}\label{eqn2.boundarysigma3}
\sigma_{3}[D^{2}u]\geq-C\quad\text{on }\partial\Omega.
\end{equation}

\subsection{The global lower bound}

For distinct indices $i_{1},\ldots,i_{r}$, write
\[
\sigma_{m;i_{1}\cdots i_{r}}(\lambda)
:=\left.\sigma_{m}(\lambda)\right|_{\lambda_{i_{1}}=\cdots=\lambda_{i_{r}}=0}.
\]
We set $\sigma_{0}=1$, set $\sigma_{m}=0$ outside its degree range, and set the deleted-component expression to zero if an index is repeated.
At $A=\operatorname{diag}\lambda$, for every symmetric matrix $T$,
\begin{align*}
\sigma_{k}^{ij}[A]=\sigma_{k-1;i}(\lambda)\delta_{ij},\quad
\sigma_{k}^{ij,st}[A]T_{ij}T_{st}=\sum_{i,j=1}^{n}\sigma_{k-2;ij}(\lambda)
\left(T_{ii}T_{jj}-T_{ij}^{2}\right).
\end{align*}

Set $q=-\sigma_{3}[D^{2}u]$.
Since $u\in C^{3,1}$ and $f\in C^{1,1}$, the differentiated identities hold almost everywhere and $q\in C^{1,1}$.
At such a point, diagonalize $D^{2}u$ with eigenvalues $\lambda=(\lambda_{1},\ldots,\lambda_{n})$.
In the algebraic formulas below, the argument $\lambda$ is omitted from $\sigma_{j}$ and $\sigma_{j;i}$.
In this frame, $\sigma_{2}^{ii}=\sigma_{1;i}>0$ and $\sigma_{3}^{ii}=\sigma_{2;i}$, and we have 
\[
\sum_{i=1}^{n}\sigma_{1;i}\lambda_{i}^{2}=\sigma_{2}\sigma_{1}+3q,
\]
which yields
\begin{equation}\label{eqn2.newton}
q=\frac{\sigma_{1}^{3}}{3}-\sigma_{2}\sigma_{1}
-\frac{1}{3}\sum_{i=1}^{n}\lambda_{i}^{3}\leq\frac{2}{3}\sigma_{1}^{3}.
\end{equation}
For each $\ell=1,\ldots,n$, twice differentiating the equation gives
\begin{align*}
\sigma_{2}^{ij}u_{ij\ell\ell}
=f_{\ell\ell}-\sigma_{2}^{ij,st}u_{ij\ell}u_{st\ell}=f_{\ell\ell}-(\Delta u_{\ell})^{2}
+\sum_{i,j=1}^{n}u_{ij\ell}^{2}.
\end{align*}
For every symmetric matrix $P$, we have
\begin{align*}
\sigma_{3}^{ij,st}P_{ij}P_{st}
&=(\Delta u)\bigl[(\operatorname{tr}P)^{2}-|P|^{2}\bigr]
-2(\operatorname{tr}P)u_{ij}P_{ij}
+2u_{ij}P_{i\ell}P_{j\ell}.
\end{align*}
Combining these identities and grouping equal third derivatives gives
\begin{equation}\label{eqn2.decomposition}
\sigma_{2}^{ij}q_{ij}\geq-\sum_{i=1}^{n} \sigma_{2;i}f_{ii}+\sum_{\ell=1}^{n}Q_{\ell}((u_{ii\ell})_{i}),
\end{equation}
where
\begin{align*}
Q_{\ell}(X)&=(\sigma_{1;\ell}^{2}-\sigma_{2})X_{\ell}^{2}
+\sum_{i\neq\ell}(4\sigma_{1;\ell}\sigma_{1;i}-\sigma_{1;\ell}^{2}-3\sigma_{2})X_{i}^{2}\\
&\quad+(\sigma_{1;\ell}^{2}+\sigma_{2})\left(\sum_{i=1}^{n}X_{i}\right)^{2}
-2\sigma_{1;\ell}\left(\sum_{i=1}^{n}X_{i}\right)
\left(\sum_{i=1}^{n}\sigma_{1;i}X_{i}\right).
\end{align*}
The omitted terms have three distinct indices and coefficient
\begin{align*}
2\sigma_{2}+2\left(\sum_{p\notin\{i,j,k\}}\lambda_{p}\right)^{2}
+4\sum_{p\notin\{i,j,k\}}\lambda_{p}^{2}&\geq0.
\end{align*}
The following inequality controls the remaining third derivative terms.

\begin{lemma}\label{thm2.algebra}
Let $n\geq3$, let $\lambda\in\Gamma_{2}$, and set $q=-\sigma_{3}(\lambda)$.
There is a constant $C_{n}$, depending only on $n$, such that, if $q\geq C_{n}\sigma_{2}^{3/2}$, then, for every $\ell\in\{1,\dots,n\}$ and $X\in\mathbb{R}^{n}$,
\begin{equation}\label{eqn2.algebra}
Q_{\ell}(X)\geq-C_{n}\frac{\sigma_{1;\ell}}{q}
\left(\sum_{i=1}^{n}\sigma_{2;i}X_{i}\right)^{2}
-C_{n}\left(\sum_{i=1}^{n}\sigma_{1;i}X_{i}\right)^{2}.
\end{equation}
\end{lemma}

\begin{proof}
Choose $C_{n}$ with $C_{n}\geq8n-2$ and $4C_{n}^{-4/3}\leq[8n(n-1)]^{-1}$.
If $\lambda_{\ell}$ is a largest component and $\sigma_{2}\leq \sigma_{1;\ell}^{2}$, then $\sigma_{1;i}\geq \sigma_{1;\ell}$ for $i\neq\ell$.
The diagonal coefficients in $Q_{\ell}$ are nonnegative, and completing the square in $\sum_{i=1}^{n}X_{i}$ gives $Q_{\ell}\geq-\left(\sum_{i=1}^{n} \sigma_{1;i}X_{i}\right)^{2}$.

For the remaining cases, delete the $\ell$th component, reindex the coordinates from $1$ to $n-1$, and set
\begin{align*}
\alpha&=\frac{(\lambda_{i})_{i\neq\ell}}{\sigma_{1;\ell}},\quad z=(X_{i})_{i\neq\ell},\quad
h=\frac{1-|\alpha|^{2}}{2},\quad t=\frac{\sigma_{2}}{\sigma_{1;\ell}^{2}},
\end{align*}
and
\begin{align*}
W&=\operatorname{diag}(z)-\alpha z^{T}-z\alpha^{T}+\left(\sum_{i=1}^{n-1}z_{i}\right)\alpha\alpha^{T}.
\end{align*}
Then $\sum_{i=1}^{n-1}\alpha_{i}=1$, $\lambda_{\ell}/\sigma_{1;\ell}=t-h$, and
\begin{align*}
X_{\ell}&=\frac{\sum_{i=1}^{n} \sigma_{1;i}X_{i}}{\sigma_{1;\ell}}+\alpha\cdot z-(1+t-h)\sum_{i=1}^{n-1}z_{i}.
\end{align*}
Since $\operatorname{tr}W=(1+|\alpha|^{2})\left(\sum_{i=1}^{n-1}z_{i}\right)-2\alpha\cdot z$, expansion gives
\begin{align*}
\frac{Q_{\ell}(X)}{\sigma_{1;\ell}^{2}}
&=2\sum_{i=1}^{n-1}(1+|\alpha|^{2}-2\alpha_{i})z_{i}^{2}
+(t-1)\left|z-\left(\sum_{i=1}^{n-1}z_{i}\right)\alpha\right|^{2}\\
&\quad+\frac{1}{2}(\operatorname{tr}W)^{2}
-\frac{\sum_{i=1}^{n} \sigma_{1;i}X_{i}}{\sigma_{1;\ell}}\operatorname{tr}W.
\end{align*}
The coefficients $1+|\alpha|^{2}-2\alpha_{i}$ are nonnegative.
Thus $t\geq1$ gives $Q_{\ell}\geq-\left(\sum_{i=1}^{n} \sigma_{1;i}X_{i}\right)^{2}/2$.

It remains to consider $t<1$ when $\lambda_{\ell}$ is not a largest component.
Then $\sigma_{1;\ell}\geq \sigma_{1}/4$: if $\lambda_{\ell}>0$, a second component is no smaller, so $2\lambda_{\ell}^{2}\leq|\lambda|^{2}<\sigma_{1}^{2}$; otherwise $\sigma_{1;\ell}\geq \sigma_{1}$.
Put
\begin{align*}
\delta&=\frac{h^{2}}{3}+\frac{1}{6}|\alpha\alpha^{T}-\operatorname{diag}\alpha|^{2}.
\end{align*}
Expanding $q$ and using $q\leq2\sigma_{1}^{3}/3$ give
\begin{align*}
\frac{q}{\sigma_{1;\ell}^{3}}=\delta-th\leq\delta+t\sqrt{3\delta},\quad
\frac{\sigma_{1;\ell}^{3}t^{2}}{q}=\frac{\sigma_{2}^{2}}{\sigma_{1;\ell}q}
\leq\frac{4\sigma_{2}^{2}}{q^{4/3}}\leq\frac{1}{8n(n-1)}.
\end{align*}
If $\delta<q/(2\sigma_{1;\ell}^{3})$, these inequalities would imply $1<1/2+\sqrt{3/[16n(n-1)]}<1$.
Hence
\begin{align*}
\delta&\geq\frac{q}{2\sigma_{1;\ell}^{3}}>0,\quad
\frac{t^{2}}{\delta}\leq\frac{1}{4n(n-1)}.
\end{align*}

Set
\begin{align*}
B_{0}&=\frac{\sum_{i=1}^{n} \sigma_{2;i}X_{i}}{\sigma_{1;\ell}^{2}}-\frac{t}{2}\operatorname{tr}W-\frac{h\left(\sum_{i=1}^{n} \sigma_{1;i}X_{i}\right)}{\sigma_{1;\ell}}.
\end{align*}
Expansion gives
\begin{align*}
\frac{Q_{\ell}(X)}{\sigma_{1;\ell}^{2}}
&=|W|^{2}-2B_{0}\sum_{i=1}^{n-1}z_{i}
+t\left|z-\left(\sum_{i=1}^{n-1}z_{i}\right)\alpha\right|^{2}
-\frac{\sum_{i=1}^{n} \sigma_{1;i}X_{i}}{\sigma_{1;\ell}}\operatorname{tr}W,\\
B_{0}+\delta\sum_{i=1}^{n-1}z_{i}
&=\frac{1}{3}\langle\alpha\alpha^{T}-\operatorname{diag}\alpha-hI,W\rangle.
\end{align*}
Moreover,
\begin{align*}
\frac{1}{9}|\alpha\alpha^{T}-\operatorname{diag}\alpha-hI|^{2}
&=\frac{2}{3}\delta+\frac{n+1}{9}h^{2}\leq\frac{n+3}{3}\delta.
\end{align*}
The Cauchy--Schwarz inequality and completion of the square give
\begin{align*}
|W|^{2}-2B_{0}\left(\sum_{i=1}^{n-1}z_{i}\right)&\geq\frac{1}{2}|W|^{2}-\frac{2n}{3\delta}B_{0}^{2}.
\end{align*}
Since $(\operatorname{tr}W)^{2}\leq(n-1)|W|^{2}$ and $h^{2}\leq3\delta$,
\begin{align*}
B_{0}^{2}&\leq\frac{3}{\sigma_{1;\ell}^{4}}\left(\sum_{i=1}^{n} \sigma_{2;i}X_{i}\right)^{2}
+\frac{3}{4}(n-1)t^{2}|W|^{2}
+\frac{9\delta}{\sigma_{1;\ell}^{2}}
\left(\sum_{i=1}^{n}\sigma_{1;i}X_{i}\right)^{2},\\
\left|\frac{\sum_{i=1}^{n} \sigma_{1;i}X_{i}}{\sigma_{1;\ell}}\operatorname{tr}W\right|
&\leq\frac{1}{8}|W|^{2}+\frac{2(n-1)}{\sigma_{1;\ell}^{2}}\left(\sum_{i=1}^{n} \sigma_{1;i}X_{i}\right)^{2}.
\end{align*}
Substitution, using the bound for $t^{2}/\delta$, yields
\begin{align*}
\frac{Q_{\ell}(X)}{\sigma_{1;\ell}^{2}}
&\geq\frac{1}{4}|W|^{2}
-\frac{2n}{\delta \sigma_{1;\ell}^{4}}\left(\sum_{i=1}^{n} \sigma_{2;i}X_{i}\right)^{2}
-\frac{8n-2}{\sigma_{1;\ell}^{2}}\left(\sum_{i=1}^{n} \sigma_{1;i}X_{i}\right)^{2}.
\end{align*}
Since $\delta\geq q/(2\sigma_{1;\ell}^{3})$ and $C_{n}\geq8n-2\geq4n$, this proves \eqref{eqn2.algebra}.
\end{proof}

\begin{proof}[Proof of Theorem \ref{thm1.1}]
For $n=2$, admissibility implies convexity and $\sigma_{3}\equiv0$, so the conclusion is immediate.
We henceforth assume $n\geq3$.
Since $\sum_{i=1}^{n} \sigma_{1;i} u_{ii\ell}=f_{\ell}$ and $\sum_{i=1}^{n} \sigma_{2;i} u_{ii\ell}=-q_{\ell}$, \eqref{eqn2.decomposition} and \eqref{eqn2.algebra} imply
\begin{equation}\label{eqn2.qineq}
\sigma_{2}^{ij}q_{ij}\geq-\sum_{i=1}^{n}\sigma_{2;i}f_{ii}
-C_{n}|Df|^{2}-\frac{C_{n}}{q}\sigma_{2}^{ij}q_{i}q_{j}.
\end{equation}
This holds where $q\geq C_{n}f^{3/2}>0$.
By \eqref{eqn2.newton},
\begin{align*}
\sum_{i=1}^{n}|\sigma_{2;i}|&\leq nf+\sqrt{(n-1)\Delta u(f\Delta u+3q)}.
\end{align*}
Choose an integer $N\geq C_{n}+1$ to absorb the squared gradient term in \eqref{eqn2.qineq} and obtain
\begin{equation}\label{eqn2.powerinequality}
\sigma_{2}^{ij}(q^{N})_{ij}\geq-Cq^{N-1}(1+\Delta u+\sqrt{q\Delta u}).
\end{equation}

Translate coordinates so that $|x|\leq\operatorname{diam}(\Omega)$ and set
\begin{align*}
\kappa&=\frac{1}{4(1+\operatorname{diam}(\Omega)^{2})},\quad W=e^{\kappa|x|^{2}}q^{N}.
\end{align*}
Since $\sigma_{2}^{ij}x_{i}x_{j}\leq(n-1)\Delta u|x|^{2}$, \eqref{eqn2.powerinequality} gives
\begin{align*}
\sigma_{2}^{ij}W_{ij}-4\kappa \sigma_{2}^{ij}x_{j}W_{i}
&\geq W\left[\kappa(n-1)\Delta u
-C\left(\frac{1}{q}+\frac{\Delta u}{q}+\sqrt{\frac{\Delta u}{q}}\right)\right].
\end{align*}
For $q\geq1$, \eqref{eqn2.newton} implies
\begin{align*}
\frac{1}{q}+\frac{\Delta u}{q}+\sqrt{\frac{\Delta u}{q}}&\leq Cq^{-2/3}\Delta u.
\end{align*}
The bracket is positive for sufficiently large $q$.
Choose $Q$ above this threshold, $C_{n}\|f\|_{\infty}^{3/2}$, and the boundary bound in \eqref{eqn2.boundarysigma3}.
The set $\{q>Q\}$ is compactly contained in $\Omega$.
For each fixed solution, the coefficients are bounded and uniformly elliptic on its closure, and $W\in C^{1,1}$.
On its boundary, $q=Q$ and $W\leq e^{1/4}Q^{N}$.
Applying the weak maximum principle \cite[Theorem 9.1]{GT} to $W$ for the operator
$\sigma_{2}^{ij}\partial_{ij}-4\kappa\sigma_{2}^{ij}x_{j}\partial_{i}$, we obtain
\[
\sup_{\{q>Q\}}W\leq e^{1/4}Q^{N}.
\]
Since $q^{N}\leq W$ on $\{q>Q\}$ and $q\leq Q$ elsewhere, this gives
\[
\sup_{\Omega}q_{+}\leq e^{1/(4N)}Q\leq C.
\]
The ellipticity constants do not enter this bound.

If $\lambda_{\min}=-c<0$, then $c\leq \Delta u$ and \eqref{eqn2.newton} gives
\begin{align*}
(\Delta u)c^{2}&\leq\sum_{i=1}^{n}\sigma_{1;i}\lambda_{i}^{2}
=f\Delta u+3q\leq f\Delta u+3q_{+},\\
c^{3}&\leq fc+3q_{+}.
\end{align*}
Consequently,
\begin{align*}
\lambda_{\min}(D^{2}u)&\geq-\max\left\{\sqrt{2\|f\|_{\infty}},(6\|q_{+}\|_{\infty})^{1/3}\right\}.
\end{align*}

\end{proof}

The semiconvexity bound passes to locally uniform limits by the second-difference inequality on line segments in $\Omega$.

\section{Global \texorpdfstring{$C^{2}$}{C2} estimates}\label{sec3}

Throughout this section, $k=2$.
We establish the global $C^{2}$ estimate used to prove Theorem \ref{thm1.2}.
We first consider $u\in C^{4}(\overline{\Omega})$ with $D^{2}u\in\Gamma_{2}$ on $\overline{\Omega}$.
Section 2 gives $D^{2}u\geq-KI$, with $K$ controlled by the data.
Unless stated otherwise, all constants depend only on the data in Theorem \ref{thm1.2}.

If $\lambda_{1}\geq\cdots\geq\lambda_{n}$ are the Hessian eigenvalues, semiconvexity and \eqref{eqn2.elementary} give
\[
[\lambda_{1}-(n-2)K](\Delta u-\lambda_{1})
\leq f+\binom{n-1}{2}K^{2}.
\]
Thus, when $\Delta u$ is large, $\lambda_{1}\asymp\Delta u$ and $|\lambda_{j}|\leq C$ for $j\geq2$.
The resulting linear growth bound controls the errors from differentiating the normal field:
\begin{equation}\label{eqn3.spectral}
\left|(\sigma_{2}^{ij}[D^{2}u])D^{2}u\right|\leq C\Delta u.
\end{equation}
We use the global $C^{1}$ and tangential and mixed boundary bounds \eqref{eqn2.first} and \eqref{eqn2.boundaryblocks}.
The global reduction \cite[Proposition 3.1]{JW}, whose assumptions on $f$ reduce to $|Df|\leq C$ and $\Delta f\geq-C$ when $k=2$, gives
\begin{equation}\label{eqn3.reduction}
\sup_{\Omega}\Delta u
\leq C\left(1+\sup_{\partial\Omega}(u_{\nu\nu})_{+}\right).
\end{equation}
Thus it remains to bound $u_{\nu\nu}$ on $\partial\Omega$.
Using semiconvexity, we first estimate this derivative in terms of interior second derivatives.
A weak interior estimate then closes the argument.

\subsection{Boundary reduction and barriers}

Let $d(x):=\operatorname{dist}(x,\partial\Omega)$.
In a fixed interior neighborhood of $\partial\Omega$, set
\[
\nu:=Dd,\quad P:=I-\nu\otimes\nu,\quad \pi(x):=x-d\nu,\quad A:=-D^{2}d,\quad \tau:=\operatorname{tr}A.
\]
Here $\nu$ is the inward normal, $P$ is the tangent projection, and $A$ is the second fundamental form of the parallel hypersurfaces.
Strict mean convexity gives $\tau\geq\tau_{*}>0$ after reducing the width of this neighborhood; $A$ need not be positive definite.
For scalars and tangent vector fields, $D_{T}v=PDv$ and $D_{T}V=P(DV)P$ denote their intrinsic derivatives.
Consequently,
\[
D_{T}^{2}v=PD^{2}vP+v_{\nu}A,\quad
\Delta_{T}v=\Delta v-v_{\nu\nu}+\tau v_{\nu}.
\]
The product rules used below are
\[
\partial_{k}P_{ij}=A_{ik}\nu_{j}+\nu_{i}A_{jk},\quad \partial_{k}A_{ij}=\partial_{i}A_{kj},\quad \partial_{\nu}A=A^{2},\quad \tau_{\nu}=|A|^{2}.
\]
Greek indices refer to an orthonormal tangential frame.
The entries $u_{\alpha\beta}$ and $u_{\alpha\nu}$ are contractions of the Euclidean Hessian, whereas $u_{\nu}=Du\cdot\nu$.
The equation and admissibility give
\begin{equation}\label{eqn3.blocks}
\begin{cases}
f=\sigma_{2}^{\nu\nu}u_{\nu\nu}
+\sigma_{2}[(u_{\alpha\beta})]-|(u_{\alpha\nu})|^{2}\leq(\Delta u)\sigma_{2}^{\nu\nu},\\
|(u_{\alpha\beta})|^{2}+2|(u_{\alpha\nu})|^{2}
\leq2(\Delta u)\sigma_{2}^{\nu\nu},\\
0<\sigma_{2}^{\nu\nu}<2\Delta u.
\end{cases}
\end{equation}

Taylor's inequality for the nonnegative right-hand side on parallel hypersurfaces and inward normal segments gives
\begin{equation}\label{eqn3.rhs}
|D_{T}f|^{2}\leq Cf,\quad
(-f_{\nu})_{+}\leq C\sqrt f\leq C\Delta u.
\end{equation}
Only values of $f$ in $\overline{\Omega}$ enter these inequalities.
Extend the boundary tangent gradient normally by $Y=D_{T}\varphi\circ\pi$.
Thus $Y\cdot\nu=0$ and $Y_{\nu}=0$.
In this boundary neighborhood, put
\[
m:=\frac{\operatorname{div}_{T}Y}{\tau},\quad q:=m-u_{\nu},
\]
where $\operatorname{div}_{T}Y=\operatorname{tr}(P(DY)P)$.
On $\partial\Omega$, set
\[
C_{0}:=D_{T}^{2}\varphi-mA,\quad p_{0}:=D_{T}m+AD_{T}\varphi,
\]
and extend $C_{0},p_{0}$ normally.
The boundary identities are
\begin{equation}\label{eqn3.compatibility}
(u_{\alpha\beta})=C_{0}+qA,\quad \sigma_{2}^{\nu\nu}=\tau q>0,\quad
(u_{\alpha\nu})=p_{0}-D_{T}q,\quad \operatorname{tr}C_{0}=0.
\end{equation}
Since $Y\in C^{2,1}$ and $\tau\geq\tau_{*}>0$, the defining formula for $m$ gives a controlled $C^{1,1}$ norm.
Its values and first tangential derivatives differ from their normal extensions by $O(d)$.

We now show that $q$ controls $f$ and the squared boundary terms:
\begin{equation}\label{eqn3.weighted}
f+|(u_{\alpha\beta})|^{2}+|C_{0}|^{2}+|p_{0}|^{2}
\leq Cq\quad\text{on }\partial\Omega.
\end{equation}
The bound $\sigma_{2}[(u_{\alpha\beta})]\geq f-C\sigma_{2}^{\nu\nu}$ from Section 2 and \eqref{eqn3.compatibility} already control the first three terms.
It remains to estimate $|p_{0}|^{2}$.
Fix $z\in\partial\Omega$, and let $x_{z}$ minimize $q(x)+|x-z|^{2}$ on the boundary.
Then
\[
q(x_{z})\leq q(z),\quad |x_{z}-z|^{2}\leq q(z),\quad
|D_{T}q(x_{z})|\leq2\sqrt{q(z)}.
\]
By the minimizing property, the function
\[
m(x)-q(x_{z})+|x-z|^{2}-|x_{z}-z|^{2}
\]
supports $u_{\nu}$ from above at $x_{z}$ and has a uniformly bounded $C^{1,1}$ norm.
In a fixed graph chart with $e_{n}=\nu(x_{z})$ and $\nu\cdot e_{n}>0$, the identity $u_{n}=Y_{n}+u_{\nu}\nu_{n}$ converts this into a $C^{1,1}$ upper support for $u_{n}$.
Estimate \eqref{eqn2.supportcontact} gives $u_{\nu\nu}(x_{z})\leq C$.
Equations \eqref{eqn3.blocks} and \eqref{eqn3.compatibility}, together with the bound for $D_{T}q(x_{z})$, give
\[
|(u_{\alpha\nu})(x_{z})|^{2}+|p_{0}(x_{z})|^{2}\leq Cq(z).
\]
Since $p_{0}$ is Lipschitz and $|x_{z}-z|^{2}\leq q(z)$, the same bound holds for $|p_{0}(z)|^{2}$, proving \eqref{eqn3.weighted}.

Since $\operatorname{tr}C_{0}=0$, the boundary equation becomes
\[
\tau qu_{\nu\nu}
=f+\frac{1}{2}|C_{0}|^{2}+|(u_{\alpha\nu})|^{2}+q(A:C_{0})-q^{2}\sigma_{2}[A],
\]
which implies 
\begin{equation}\label{eqn3.boundary_equation}
c\frac{|(u_{\alpha\nu})|^{2}}{q}-C\leq u_{\nu\nu}\leq C\left(1+\frac{|(u_{\alpha\nu})|^{2}}{q}\right).
\end{equation}
Thus the remaining boundary estimate is equivalent to a bound for $|(u_{\alpha\nu})|^{2}/q$, uniform as $q$ approaches zero.
We use a positive denominator $Z$ in this neighborhood, with boundary value $q$.
First we control $\sigma_{2}^{\nu\nu}/Z$, whose boundary value is $\tau$; then we estimate the mixed derivative quotient.

Choose a fixed large $N$ and a fixed small $T_{0}>0$ so that $\rho:=d-Nd^{2}$ satisfies
\begin{equation}\label{eqn3.barrier}
\frac{3}{4}d\leq\rho\leq d,\quad \frac{1}{2}\leq\rho'\leq1,\quad
\sigma_{2}^{ij}\rho_{ij}\leq-c\Delta u \quad(0\leq d\leq2T_{0}).
\end{equation}
Indeed, at $d=0$ the Hessian of $-\rho$ has blocks $A$ and $2N$, so $\operatorname{tr}D^{2}(-\rho)-|D^{2}(-\rho)|\geq c>0$ for large $N$.
This inequality persists in the boundary neighborhood and implies \eqref{eqn3.barrier} by $|D^{2}u|\leq\Delta u$.
Fix the geometric constants in this boundary neighborhood.
All smallness conditions on $T$ below are imposed before $\Lambda$ is chosen.
By \eqref{eqn2.first}, \eqref{eqn3.spectral}, and \eqref{eqn3.rhs},
\[
|\sigma_{2}^{ij}(PDu-Y)_{ij}|\leq C\Delta u.
\]
Comparison with a fixed multiple of $\rho$ in the boundary neighborhood gives
\begin{equation}\label{eqn3.W}
W:=\frac{Y-PDu}{d},\quad |W|\leq C.
\end{equation}
This quotient extends to the boundary, where $W=-(u_{\alpha\nu})$.

For a fixed sufficiently small $T\leq T_{0}$ and every $\Lambda\geq\Lambda_{0}$, put
\[
Z:=q-u+\widetilde\varphi+\Lambda\rho,\quad
\widetilde\varphi:=\varphi\circ\pi.
\]
We use the following bounds for $Z$ in both comparisons:
\begin{align}
&Z|_{d=0}=q,\quad 
Z\geq c(f+|C_{0}|^{2}+|p_{0}|^{2}+\Lambda d),\notag\\
&\sigma_{2}^{ij}Z_{ij}\leq-c\Lambda\Delta u,\label{eqn3.Zbounds}\\
&|q|\leq C,\quad q\geq-Cd,\notag\\
&-\frac{C}{\Lambda}\leq\frac{q}{Z}\leq1,\quad
Z\leq C(1+\Lambda T),\label{eqn3.qbounds}\\
&|D_{T}Z+(u_{\alpha\nu})|^{2}
+|D_{T}q+(u_{\alpha\nu})|^{2}\leq CZ,\label{eqn3.tangenterrors}\\
&c\Lambda\leq Z_{\nu}+u_{\nu\nu}\leq C\Lambda.\label{eqn3.normalerror}
\end{align}
For the differential bounds, use $Dq=Dm-D^{2}u\,\nu+APDu$ and
\[
\sigma_{2}^{ij}(u_{\nu})_{ij}
=f_{\nu}+2\sigma_{2}^{ij}u_{jk}d_{ki}
+Du\cdot(\sigma_{2}^{ij}d_{kij})_{k=1}^{n}.
\]
Equations \eqref{eqn3.spectral}, \eqref{eqn3.rhs}, and \eqref{eqn3.barrier} give $\sigma_{2}^{ij}Z_{ij}\leq-c\Lambda\Delta u$ and $\sigma_{2}^{ij}(\tau q)_{ij}\leq C\Delta u$ for large $\Lambda_{0}$.
To prove the lower bound without differentiating the Lipschitz field $p_{0}$ twice, fix $z\in\partial\Omega$ and use
\[
g_{z}=f+\frac{1}{2}(|C_{0}(z)|^{2}+|p_{0}(z)|^{2})-C|\pi(x)-z|^{2},
\]
where the fixed $C$ is chosen using the Lipschitz norms of $C_{0},p_{0}$.
Then $g_{z}\leq f+|C_{0}|^{2}+|p_{0}|^{2}$ on the boundary and $|\sigma_{2}^{ij}(g_{z})_{ij}|\leq C\Delta u$ in the boundary neighborhood.
By \eqref{eqn3.weighted}, $q\geq c g_{z}$ on the boundary.
At $\Lambda=\Lambda_{0}$, comparison with $cg_{z}+c'\Lambda_{0}d$, for fixed small $c,c'>0$, gives $Z\geq cg_{z}+c'\Lambda_{0}d$.
Choosing $z=\pi(x)$ after each fixed comparison and adding $(\Lambda-\Lambda_{0})\rho$ proves the lower bound in \eqref{eqn3.Zbounds}.
The remaining bounds follow from \eqref{eqn2.first}, \eqref{eqn3.W}, and the identities
\[
D_{T}q+(u_{\alpha\nu})=D_{T}m+APDu=p_{0}+O(d),\quad
D_{T}Z+(u_{\alpha\nu})=p_{0}+O(d).
\]
The bound $q\geq-Cd$ is fixed at an initial value of $\Lambda_{0}$ in the original boundary neighborhood and is retained when $T$ is reduced or $\Lambda_{0}$ is enlarged.

The distance is $C^{3,1}$; $\nu,P,\pi,Y$ are $C^{2,1}$; $A,\tau,m,C_{0}$ are $C^{1,1}$; and $p_{0}$ is Lipschitz.
All differentiated identities below hold almost everywhere.
For each fixed strictly admissible solution, $(\sigma_{2}^{ij})$ is uniformly positive definite on compact subsets of the boundary neighborhood.
The exact quotient identities retain the gradient terms and, above the stated thresholds, give locally an inequality of the form
\[
\sigma_{2}^{ij}h_{ij}\geq c-C(|Dh|+|Dh|^{2}),\quad c>0.
\]
For $h\in C^{1,1}$, the gradient terms form a locally bounded drift, and the maximum principle \cite[Chapter 9]{GT} excludes an interior maximum.
The ellipticity constants enter only this qualitative comparison.
The physical boundary calculations use continuous first derivatives.

\subsection{Tangential trace comparison}

The quantity $\tau q$ agrees with the tangential trace on the boundary.
The comparison below gives the upper bound $\sigma_{2}^{\nu\nu}/Z\leq\tau+\eta_{0}$ needed to control the curvature term in the mixed derivative estimate.
Set $d_{0}:=\sigma_{2}^{\nu\nu}-\tau q$.

\begin{lemma}\label{lem3.trace}
There exist $T>0$ and $\Lambda_{0}>0$, depending only on the data in Theorem \ref{thm1.2}, such that, for every $\Lambda\geq\Lambda_{0}$,
\[
\sup_{0<d\leq T}\frac{(d_{0}^{+})^{2}}{\rho Z}
\leq C+\frac{C}{\Lambda T^{2}}
\left(1+\sup_{\{d=T\}}\sigma_{2}^{\nu\nu}\right)^{2}.
\]
Here $C$ depends only on the same data and is independent of $\Lambda$.
\end{lemma}

\begin{proof}
Put $h:=(d_{0}^{+})^{2}/(\rho Z)$.
For each fixed strictly admissible solution, $Z=q>0$ at $d=0$ and
$d_{0}=O(d)$, so $h\to0$ there.
At $d=T$, the lower bound for $Z$ gives
\[
h\leq\frac{C}{\Lambda T^{2}}
\left(1+\sup_{\{d=T\}}\sigma_{2}^{\nu\nu}\right)^{2}.
\]
We exclude interior maxima above a fixed threshold by retaining the positive third derivative squares.
At the point under consideration, choose Cartesian coordinates with $e_{n}=\nu$; all third derivatives below are ordinary Cartesian derivatives.
For each fixed $a<n$, define the symmetric matrix $\mathcal A_{a}$ by
\[
\begin{aligned}
(\mathcal A_{a})_{\alpha\beta}
&:=u_{\alpha\beta a}+A_{\alpha a}u_{\beta\nu}+A_{\beta a}u_{\alpha\nu},\\
(\mathcal A_{a})_{\alpha\nu}
&:=u_{\alpha\nu a}+u_{\nu\nu}A_{\alpha a}-u_{\alpha\beta}A_{\beta a},\\
(\mathcal A_{a})_{\nu\nu}
&:=u_{\nu\nu a}-2u_{\alpha\nu}A_{\alpha a}.
\end{aligned}
\]
These are the derivatives of the Hessian in the moving tangential and normal frame.
Thus $\sigma_{2}^{ij}(\mathcal A_{a})_{ij}=f_{a}$ and
$\operatorname{tr}(\mathcal A_{a})_{TT}=\partial_{a}\sigma_{2}^{\nu\nu}$.
Separate the part proportional to the tangential trace derivative by setting
\[
(X_{a})_{\alpha\beta}
:=(\mathcal A_{a})_{\alpha\beta}
-\frac{\partial_{a}\sigma_{2}^{\nu\nu}}{\sigma_{2}^{\nu\nu}}u_{\alpha\beta},\quad
(Y_{a})_{\alpha}
:=(\mathcal A_{a})_{\alpha\nu}
-\frac{\partial_{a}\sigma_{2}^{\nu\nu}}{\sigma_{2}^{\nu\nu}}u_{\alpha\nu},\quad
\mathcal Y:=(Y_{1},\ldots,Y_{n-1}).
\]
Eliminating the normal-normal entry with $\sigma_{2}^{ij}(\mathcal A_{a})_{ij}=f_{a}$ gives
\[
-\sigma_{2}^{ij,st}(\mathcal A_{a})_{ij}(\mathcal A_{a})_{st}
=|X_{a}|^{2}+2|Y_{a}|^{2}+\frac{2f(\partial_{a}\sigma_{2}^{\nu\nu})^{2}}{(\sigma_{2}^{\nu\nu})^{2}}
-\frac{2(\partial_{a}\sigma_{2}^{\nu\nu})f_{a}}{\sigma_{2}^{\nu\nu}}.
\]
The bound $|D_{T}f|^{2}\leq Cf$ absorbs the last two terms and gives
\[
-\sum_{a=1}^{n-1}\sigma_{2}^{ij,st}(\mathcal A_{a})_{ij}(\mathcal A_{a})_{st}
\geq\sum_{a=1}^{n-1}|X_{a}|^{2}+2|\mathcal Y|^{2}-C.
\]

Differentiating $\sigma_{2}^{\nu\nu}=P_{ij}u_{ij}$ and retaining all third derivatives of $u$ gives
\[
\begin{aligned}
\sigma_{2}^{ij}(\sigma_{2}^{\nu\nu})_{ij}
&\geq -\sum_{a=1}^{n-1}\sigma_{2}^{ij,st}(\mathcal A_{a})_{ij}(\mathcal A_{a})_{st}+4\sigma_{2}^{ij}\nu_{i}\partial_{j}(A_{\alpha\beta}u_{\alpha\beta})
+2(\Delta u)u_{\alpha\nu}(D_{T}\tau)_{\alpha}\\
&\hspace{4.5mm}-C(\Delta u)\sigma_{2}^{\nu\nu}-C.
\end{aligned}
\]
The lower-order terms are bounded by $C(\Delta u)\sigma_{2}^{\nu\nu}$
using \eqref{eqn3.blocks} and $|u_{\nu\nu}|\leq\Delta u$;
$|P_{ij}f_{ij}|\leq C$.
The identities
\[
\partial_{\nu}(A_{\alpha\beta}u_{\alpha\beta})
=A_{\alpha\beta}u_{\nu\alpha\beta}+(A^{2})_{\alpha\beta}u_{\alpha\beta},\quad
(D_{T})_{a}(A_{\alpha\beta}u_{\alpha\beta})
=A_{\alpha\beta}(\mathcal A_{a})_{\alpha\beta}
+u_{\alpha\beta}(D_{T})_{a}A_{\alpha\beta}
\]
give
\[
4\sigma_{2}^{ij}\nu_{i}\partial_{j}(A_{\alpha\beta}u_{\alpha\beta})
\geq 4\sigma_{2}^{\nu\nu}A:\mathcal Y-4\sum_{a=1}^{n-1}u_{a\nu}(A:X_{a})+\zeta\cdot D_{T}\sigma_{2}^{\nu\nu}-C(\Delta u)\sigma_{2}^{\nu\nu},
\]
where
\[
\zeta_{\alpha}:=4\left(A_{\alpha\beta}u_{\beta\nu}
-\frac{A_{\gamma\delta}u_{\gamma\delta}}{\sigma_{2}^{\nu\nu}}u_{\alpha\nu}\right).
\]
Completing the $X_{a}$ and $\mathcal Y$ squares retains $|\mathcal Y|^{2}$.
Subtracting the bound for $\sigma_{2}^{ij}(\tau q)_{ij}$ gives
\begin{equation}\label{eqn3.trace_differential}
\sigma_{2}^{ij}(d_{0})_{ij}
\geq |\mathcal Y|^{2}+\zeta\cdot D_{T}\sigma_{2}^{\nu\nu}
+2(\Delta u)u_{\alpha\nu}(D_{T}\tau)_{\alpha}-C(\Delta u)(\sigma_{2}^{\nu\nu}+1)-C.
\end{equation}

To control the first derivative term, we use $|\zeta|+|\zeta_{\alpha} u_{\alpha\nu}|\leq C\Delta u$.
The first bound follows from \eqref{eqn3.blocks}.
For the second, if $\sigma_{2}^{\nu\nu}\leq1$, semiconvexity bounds
$(u_{\alpha\beta})$ and $\sum_{\alpha=1}^{n-1} u_{\alpha\nu}^{2}\leq(\Delta u)\sigma_{2}^{\nu\nu}$.
If $\sigma_{2}^{\nu\nu}\geq1$, diagonalize $(u_{\alpha\beta})$.
The bounded spectral tail gives $u_{11}=\sigma_{2}^{\nu\nu}+O(1)$ and
$A_{\alpha\beta}u_{\alpha\beta}/\sigma_{2}^{\nu\nu}=A_{11}+O((\sigma_{2}^{\nu\nu})^{-1})$.
The identity
\[
u_{\alpha\nu}u_{\beta\nu}
=(\Delta u)u_{\alpha\beta}
-u_{\alpha\gamma}u_{\gamma\beta}-\sigma_{2}^{\alpha j}u_{j\beta}
\]
and its tangential trace cancel the $\Delta u$ terms in $\zeta_{\alpha} u_{\alpha\nu}/4$.
The sole unbounded square has coefficient $O((\sigma_{2}^{\nu\nu})^{-1})$, so it contributes at most $C(\sigma_{2}^{\nu\nu}+1)\leq C\Delta u$.
The remaining terms are bounded by \eqref{eqn3.spectral}.

For $h=d_{0}^{2}/(\rho Z)>0$, the exact identities before imposing stationarity are
\[
Dd_{0}=\frac{d_{0}}{2}\left(\frac{D\rho}{\rho}+\frac{DZ}{Z}\right)
+\frac{d_{0}}{2h}Dh
\]
and
\[
\begin{aligned}
\rho \sigma_{2}^{ij}h_{ij}
&=\frac{2d_{0}}{Z}\sigma_{2}^{ij}(d_{0})_{ij}
-h\sigma_{2}^{ij}\rho_{ij}
-\frac{\rho h}{Z}\sigma_{2}^{ij}Z_{ij}+\frac{\rho h}{2}\sigma_{2}^{ij}
\left(\frac{Z_{i}}{Z}-\frac{\rho_{i}}{\rho}\right)
\left(\frac{Z_{j}}{Z}-\frac{\rho_{j}}{\rho}\right)\\
&\hspace{4.5mm}-\frac{1}{Z}\sigma_{2}^{ij}h_{i}(\rho Z)_{j}
+\frac{\rho}{2h}\sigma_{2}^{ij}h_{i} h_{j}.
\end{aligned}
\]
At a positive interior maximum, $Dh=0$ and $Z\geq c\Lambda\rho$ give $d_{0}/\rho\geq c\sqrt{\Lambda h}$.
For a sufficiently large threshold, $q\geq-Cd$ and $|q|\leq C$ imply
$d_{0}\leq2\sigma_{2}^{\nu\nu}$ and $\sigma_{2}^{\nu\nu}\leq d d_{0}/\rho+C$.
Moreover,
\[
D_{T}\sigma_{2}^{\nu\nu}
+\left(\frac{d_{0}}{2Z}+\tau\right)(u_{\alpha\nu})=\frac{d_{0}}{2Z}\bigl(D_{T}Z+(u_{\alpha\nu})\bigr)+qD_{T}\tau+\tau\bigl(D_{T}q+(u_{\alpha\nu})\bigr),
\]
\[
\left|D_{T}\sigma_{2}^{\nu\nu}
+\left(\frac{d_{0}}{2Z}+\tau\right)(u_{\alpha\nu})\right|^{2}
\leq C\left(\frac{d_{0}^{2}}{Z}+Z\right).
\]
The bounded errors in \eqref{eqn3.tangenterrors} give
$\zeta\cdot D_{T}\sigma_{2}^{\nu\nu}\geq-C(\Delta u)(1+d_{0}/Z)$.
By \eqref{eqn3.barrier} and \eqref{eqn3.Zbounds},
\[
-\sigma_{2}^{ij}\rho_{ij}-\frac{\rho}{Z}\sigma_{2}^{ij}Z_{ij}\geq c\Delta u.
\]
The square in the quotient identity is nonnegative.

We estimate $\mathcal Y$ through its trace.
The normal derivative relation
\[
\partial_{\nu} \sigma_{2}^{\nu\nu}+\tau u_{\nu\nu}
=(d_{0})_{\nu}+\tau_{\nu} q+\tau m_{\nu}
\]
combined with stationarity and \eqref{eqn3.blocks}, gives
\[
\begin{aligned}
\operatorname{tr}\mathcal Y&=\frac{\rho'd_{0}}{2\rho}
+\frac{d_{0}}{2Z}(Z_{\nu}+u_{\nu\nu})+\frac{d_{0}\bigl(\sigma_{2}[(u_{\alpha\beta})]-f\bigr)}{2Z\sigma_{2}^{\nu\nu}}
+\frac{\tau}{\sigma_{2}^{\nu\nu}}\sum_{\alpha=1}^{n-1} u_{\alpha\nu}^{2}
-A_{\alpha\beta}u_{\alpha\beta}\\
&\quad+\tau_{\nu} q+\tau m_{\nu}-\frac{(u_{\alpha\nu})}{\sigma_{2}^{\nu\nu}}\cdot
\left[D_{T}\sigma_{2}^{\nu\nu}
+\left(\frac{d_{0}}{2Z}+\tau\right)(u_{\alpha\nu})\right].
\end{aligned}
\]
Completion of the last term against half of
$\tau\sum_{\alpha=1}^{n-1} u_{\alpha\nu}^{2}/\sigma_{2}^{\nu\nu}$ costs at most
\[
C\frac{d_{0}^{2}}{\sigma_{2}^{\nu\nu}Z}+C\frac{Z}{\sigma_{2}^{\nu\nu}}
\leq C\frac{d_{0}}{\rho}\left(\frac{1}{\Lambda}+\frac{1}{h}\right).
\]
Also $d_{0}f/(2Z\sigma_{2}^{\nu\nu})\leq C$, and semiconvexity gives
\[
\sigma_{2}[(u_{\alpha\beta})]^{-}+|A_{\alpha\beta}u_{\alpha\beta}|
\leq C(\sigma_{2}^{\nu\nu}+1).
\]
When $\Delta u<1$, the trace already gives a contradiction, since
$|\sigma_{2}[(u_{\alpha\beta})]-f|\leq2\sigma_{2}^{\nu\nu}$.
The trace identity and completion above give
\[
\operatorname{tr}\mathcal Y
\geq\frac{d_{0}}{\rho}
\left(\frac{\rho'}{2}-\frac{C}{\Lambda}-\frac{C}{h}\right)-C.
\]
For sufficiently large $\Lambda_{0}$ and comparison threshold, $|\mathcal Y|^{2}\geq c(d_{0}/\rho)^{2}$.
The other terms in \eqref{eqn3.trace_differential} cost at most
$C(1+d_{0}/(\Lambda\rho))$, so $\sigma_{2}^{ij}(d_{0})_{ij}>0$.
This contradicts the quotient identity.

For $\Delta u\geq1$, combining the quotient identity with \eqref{eqn3.trace_differential} gives
\begin{equation}\label{eqn3.trace_quotient}
0\geq\frac{2d_{0}}{Z}|\mathcal Y|^{2}+ch\Delta u-C\frac{2d_{0}}{Z}(\Delta u)
\left(\sigma_{2}^{\nu\nu}+1+\sqrt{\sum_{\alpha=1}^{n-1} u_{\alpha\nu}^{2}}+\frac{d_{0}}{Z}\right).
\end{equation}
After division by $h\Delta u$, all losses except
$\sqrt{\sum_{\alpha=1}^{n-1} u_{\alpha\nu}^{2}}$ are bounded by
$C(d+\rho/d_{0}+\Lambda^{-1})$.
If $|(u_{\alpha\nu})|\leq c d_{0}/\rho$, with a sufficiently small fixed $c>0$, this loss is absorbed as well, a contradiction.

Otherwise $|(u_{\alpha\nu})|>c d_{0}/\rho$.
The possible negative contribution of $\sigma_{2}[(u_{\alpha\beta})]$ in the trace identity, divided by
$\sum_{\alpha=1}^{n-1} u_{\alpha\nu}^{2}/\sigma_{2}^{\nu\nu}$, is bounded by
\[
C\frac{d_{0}(\sigma_{2}^{\nu\nu}+1)}{Z\sum_{\alpha=1}^{n-1} u_{\alpha\nu}^{2}}
\leq C
\left(\frac{d}{\Lambda}+\frac{\rho}{\Lambda d_{0}}\right).
\]
Fix $T$ sufficiently small, then enlarge $\Lambda_{0}$ and the comparison threshold, all by data-dependent choices.
The trace identity yields
\[
\operatorname{tr}\mathcal Y
\geq\frac{d_{0}}{8\rho}
+\frac{\tau_{*}}{4\sigma_{2}^{\nu\nu}}\sum_{\alpha=1}^{n-1} u_{\alpha\nu}^{2}.
\]
The equation, $f\leq CZ$ from \eqref{eqn3.Zbounds}, and $\sigma_{2}^{\nu\nu}\leq d d_{0}/\rho+C$ also give
\[
\frac{(\sigma_{2}^{\nu\nu})^{2}}{\sum_{\alpha=1}^{n-1}u_{\alpha\nu}^{2}}
\leq C\left(d+\frac{\rho}{d_{0}}\right)^{2},\quad
\Delta u\leq\frac{3}{\sigma_{2}^{\nu\nu}}
\sum_{\alpha=1}^{n-1}u_{\alpha\nu}^{2}.
\]
Consequently, using \eqref{eqn3.blocks},
\[
|\mathcal Y|^{2}\geq c(\Delta u)^{2},\quad
\frac{\sqrt{\sum_{\alpha=1}^{n-1}u_{\alpha\nu}^{2}}}{\Delta u}
\leq\frac{\sigma_{2}^{\nu\nu}}
{\sqrt{\sum_{\alpha=1}^{n-1}u_{\alpha\nu}^{2}}}
\leq C\left(d+\frac{\rho}{d_{0}}\right).
\]
Thus the mixed loss is absorbed by half of $|\mathcal Y|^{2}$.
The remaining losses in \eqref{eqn3.trace_quotient} are absorbed by $ch\Delta u$, a contradiction.
\end{proof}

Since $q\leq Z$ by \eqref{eqn3.qbounds}, Lemma \ref{lem3.trace} implies
\[
\frac{\sigma_{2}^{\nu\nu}}{Z}
\leq\tau+\frac{C}{\sqrt\Lambda}
+\frac{C(1+\sup_{\{d=T\}}\sigma_{2}^{\nu\nu})}{\Lambda T}.
\]
Fix a sufficiently small data constant $\eta_{0}>0$.
After fixing $T$ and $\Lambda_{0}$ to satisfy the restrictions below, choose
\begin{equation}\label{eqn3.Lambda}
\Lambda=\Lambda_{0}+\frac{C}{T}
\left(1+\sup_{\{d=T\}}\sigma_{2}^{\nu\nu}\right),
\end{equation}
where $\Lambda_{0}$ and $C$ are fixed sufficiently large.
Then $\sigma_{2}^{\nu\nu}/Z\leq\tau+\eta_{0}$ throughout this boundary neighborhood.

\subsection{Boundary double normal derivative estimate}

By \eqref{eqn3.boundary_equation}, it remains to estimate the mixed derivative quotient.
Using the tangential trace comparison, we prove
\begin{equation}\label{eqn3.normal_quotient}
\sup_{0\leq d\leq T}\frac{|W|^{2}}{Z}\leq C(1+\Lambda).
\end{equation}
Put $Q=|W|^{2}/Z$.
On $d=T$, \eqref{eqn3.W} and \eqref{eqn3.Zbounds} give $Q\leq C/(\Lambda T)$.
We exclude maxima above a sufficiently large multiple of $1+\Lambda$, first on $\partial\Omega$ and then in the interior of the boundary neighborhood.

On the physical boundary,
\[
W=-(u_{\alpha\nu}),\quad W_{\nu}=-\frac{1}{2}\partial_{\nu}(u_{\alpha\nu}),\quad
\partial_{\nu}(u_{\alpha\nu})=D_{T}u_{\nu\nu}+2A(u_{\alpha\nu}).
\]
Hence
\begin{equation}\label{eqn3.normal_boundary}
qQ_{\nu}=(u_{\alpha\nu})\cdot D_{T}u_{\nu\nu}
+2A((u_{\alpha\nu}),(u_{\alpha\nu}))-Z_{\nu} Q.
\end{equation}
Here $Z=q$, so $|(u_{\alpha\nu})|^{2}/q=Q$ in \eqref{eqn3.boundary_equation}.
At a tangential maximum of $Q$,
\[
D_{T}|(u_{\alpha\nu})|^{2}=QD_{T}q,\quad
D_{T}q=p_{0}-(u_{\alpha\nu}).
\]
Differentiating the boundary equation along $(u_{\alpha\nu})$ therefore gives
\begin{align}\label{eqn3.boundary_derivative}
\tau q\,(u_{\alpha\nu})\cdot D_{T}u_{\nu\nu}&=(u_{\alpha\nu})\cdot D_{T}(f+\frac{1}{2}|C_{0}|^{2})
+q (u_{\alpha\nu})\cdot D_{T}(A:C_{0})\notag\\
&\hspace{4.5mm}-q^{2}(u_{\alpha\nu})\cdot D_{T}\sigma_{2}[A]
-qu_{\nu\nu}(u_{\alpha\nu})\cdot D_{T}\tau\notag\\
&\hspace{4.5mm}+\left(\frac{f+\frac{1}{2}|C_{0}|^{2}}{q}+q\sigma_{2}[A]\right)
\bigl(|(u_{\alpha\nu})|^{2}-(u_{\alpha\nu})\cdot p_{0}\bigr).
\end{align}
The nonnegative $C^{1,1}$ function $f+|C_{0}|^{2}/2$ on the closed boundary satisfies
\[
|D_{T}(f+\frac{1}{2}|C_{0}|^{2})|+|p_{0}|\leq C\sqrt q
\]
by \eqref{eqn3.weighted} and Taylor's inequality.
For large $Q$, $0\leq|(u_{\alpha\nu})|^{2}-(u_{\alpha\nu})\cdot p_{0}\leq q(Q+C\sqrt Q)$.
The contribution of $(f+|C_{0}|^{2}/2)/q$ in \eqref{eqn3.boundary_derivative} is nonnegative, and the remaining terms give
\[
(u_{\alpha\nu})\cdot D_{T}u_{\nu\nu}\geq-C(Q^{3/2}+Q+\sqrt Q).
\]
Thus \eqref{eqn3.normal_boundary} and \eqref{eqn3.boundary_equation} imply
\[
qQ_{\nu}\geq cQ^{2}-C(Q^{3/2}+Q+\sqrt Q)-C(1+\Lambda)Q>0
\]
when $Q$ exceeds a sufficiently large fixed multiple of $1+\Lambda$.
This contradicts the inward derivative condition at a boundary maximum.
For the interior comparison, we derive the equation for $W$ and control its error terms.
Differentiating $dW=Y-PDu$ gives
\[
(u_{\alpha\beta})=D_{T}Y-u_{\nu}A-dD_{T}W,\quad
(u_{\alpha\nu})=-W-dW_{\nu},
\]
and
\[
d\operatorname{div}_{T}W=\tau q-\sigma_{2}^{\nu\nu},\quad
|D_{T}W-(D_{T}W)^{T}|\leq C.
\]
The last estimate follows because $D_{T}Y$ is symmetric at $d=0$, its skew part is $O(d)$, and $(u_{\alpha\beta}),A$ are symmetric.
A second differentiation gives
\begin{align}\label{eqn3.W_equation}
P\bigl(d\sigma_{2}^{ij}W_{ij}+2\sigma_{2}^{ij}d_{i}W_{j}\bigr)
&=-D_{T}f-\sigma_{2}^{ij}d_{ij}W-2\sigma_{2}^{\nu\nu}A(u_{\alpha\nu})\nonumber\\
&\hspace{4.5mm}-2dA(D_{T}W)(u_{\alpha\nu})-2dA(\sigma_{2}^{\alpha\beta})AW+R,
\end{align}
\[
R_{a}=-u_{\nu} \sigma_{2}^{\alpha\beta}(D_{T})_{\alpha} A_{\beta a}
+P_{ar}\sigma_{2}^{\alpha\beta}\partial_{\alpha\beta}Y_{r}+2(A(\sigma_{2}^{\alpha\beta})AY)_{a}+2(AD_{T}Y-D_{T}YA)_{a\beta}u_{\beta\nu}.
\]
We identify the coefficient of $(\Delta u)$ in $R$:
\[
P(P:D^{2}Y)+2A^{2}Y-u_{\nu} D_{T}\tau.
\]
On the physical boundary,
\[
P(P:D^{2}Y)+2A^{2}Y
=D_{T}(\Delta_{T}\varphi)+\tau A D_{T}\varphi,
\]
so this coefficient equals $\tau(\pi(x))p_{0}+qD_{T}(\tau\circ\pi)+O(d)$ throughout this boundary neighborhood.
For the displayed identity, $\partial_{\nu} D\widetilde\varphi=AD\widetilde\varphi$, $\partial_{\nu}^{2}D\widetilde\varphi=2A^{2}D\widetilde\varphi$, and $D\widetilde\varphi=Y$ on $d=0$.
Consequently, on this boundary,
\[
P:D^{2}(Y-D\widetilde\varphi)=\tau AY,\quad
P(P:D^{2}(D\widetilde\varphi))
=D_{T}\Delta\widetilde\varphi-2A^{2}D\widetilde\varphi.
\]
The $O(d)$ remainder uses only the Lipschitz continuity of the prescribed third derivatives.
By \eqref{eqn3.blocks}, \eqref{eqn3.Zbounds}, and \eqref{eqn3.tangenterrors},
\begin{equation}\label{eqn3.W_remainder}
|R|\leq C(\Delta u)\sqrt Z+C\sqrt{(\Delta u)\sigma_{2}^{\nu\nu}}.
\end{equation}

At interior points where $W\neq0$, define
\[
M:=\frac{d}{\sigma_{2}^{\nu\nu}}D_{T}W,\quad
e:=\frac{W}{|W|},\quad z:=\frac{(u_{\alpha\nu})}{|W|},\quad
\mu:=1+\frac{dZ_{\nu}}{2Z},\quad
\epsilon:=\frac{\tau q}{\sigma_{2}^{\nu\nu}}.
\]
The first derivative identities are
\begin{equation}\label{eqn3.stationarity}
\begin{cases}
M^{T}e=\frac{d}{2\sigma_{2}^{\nu\nu}|W|}(QD_{T}Z+ZD_{T}Q),\\
e\cdot z=-\mu-\frac{d}{2Q}Q_{\nu},\\
\operatorname{tr}M=-1+\epsilon,\\
|M-M^{T}|\leq\frac{Cd}{\sigma_{2}^{\nu\nu}}.
\end{cases}
\end{equation}
At an interior maximum, $DQ=0$.
To organize the quadratic terms, set
\[
\mathcal K=|z|^{2}|M|^{2}-|Mz|^{2}
+|(I+M)z+(2I-dA)e|^{2}-1+2dA(e,e).
\]
Applying the product rule to $Q$, completing the normal derivative square, and using \eqref{eqn3.W_equation} and \eqref{eqn3.stationarity} yield
\begin{equation}\label{eqn3.normal_energy}
0\geq\frac{2\sigma_{2}^{\nu\nu}Q}{d}\mathcal K+2(A_{\alpha\beta}\sigma_{2}^{\alpha\beta})Q-Cd(\Delta u)Q-C\sigma_{2}^{\nu\nu}Q-C(\Delta u)\sqrt Q-C\sqrt{(\Delta u)\sigma_{2}^{\nu\nu}Q/Z}.
\end{equation}
Here the ambient derivative of $W$ has normal component $AW$ on tangent inputs.
After retaining $|z|^{2}|M|^{2}-|Mz|^{2}$, the remaining tangential term is nonnegative because
\[
\sigma_{2}^{\nu\nu}(\sigma_{2}^{\alpha\beta})
-|(u_{\alpha\nu})|^{2}I_{n-1}
=\left(f+\frac{(\sigma_{2}^{\nu\nu})^{2}
+|(u_{\alpha\beta})|^{2}}{2}\right)I_{n-1}-\sigma_{2}^{\nu\nu}(u_{\alpha\beta})\geq0.
\]
Indeed, after diagonalizing $(u_{\alpha\beta})$, each eigenvalue is
\[
f+\frac{1}{2}\left[
(\sigma_{2}^{\nu\nu}-u_{\alpha\alpha})^{2}
+\sum_{\beta\ne\alpha}u_{\beta\beta}^{2}
\right]\geq0
\quad\text{(no summation over $\alpha$)}.
\]
The other discarded term is $-dQ\sigma_{2}^{ij}Z_{ij}/Z\geq0$.
The remaining terms follow from \eqref{eqn3.W_remainder} and
\[
\begin{aligned}
d\sigma_{2}^{ij}Z_{ij}+2\sigma_{2}^{ij}d_{i}Z_{j}
=d\sigma_{2}^{ij}Z_{ij}+2\sigma_{2}^{\nu\nu}(Z_{\nu}+u_{\nu\nu})+2(\sigma_{2}[(u_{\alpha\beta})]-f)
-2(u_{\alpha\nu})\cdot\bigl(D_{T}Z+(u_{\alpha\nu})\bigr).
\end{aligned}
\]

When $\mu\leq1/2$, the curvature term supplies the positive margin.
By \eqref{eqn3.normalerror} and $\mu\leq1/2$, $u_{\nu\nu}\geq c\Lambda+Z/d$.
Here semiconvexity gives $|(u_{\alpha\beta})|\leq C(\sigma_{2}^{\nu\nu}+1)$.
Thus the positive mean-curvature term controls the tangential Hessian error when the normal entry dominates.
Together with $\sigma_{2}^{\nu\nu}/Z\leq\tau+\eta_{0}$, this yields, after fixing $T$ sufficiently small and $\Lambda_{0}$ sufficiently large,
\begin{equation}\label{eqn3.curvature}
A_{\alpha\beta}\sigma_{2}^{\alpha\beta}\geq c(\Delta u),\quad
\frac{d}{\sigma_{2}^{\nu\nu}}A_{\alpha\beta}\sigma_{2}^{\alpha\beta}
\geq\frac{2\tau Z}{\sigma_{2}^{\nu\nu}}(1-\mu)-Cd,\quad
-\frac{C}{\Lambda}\frac{\tau Z}{\sigma_{2}^{\nu\nu}}
\leq\epsilon\leq\frac{\tau Z}{\sigma_{2}^{\nu\nu}}.
\end{equation}
Indeed, $\sigma_{2}^{\nu\nu}/u_{\nu\nu}\leq Cd$, and
\[
\frac{d}{\sigma_{2}^{\nu\nu}}(A_{\alpha\beta}\sigma_{2}^{\alpha\beta})
=\frac{2\tau Z}{\sigma_{2}^{\nu\nu}}(1-\mu)+\frac{d}{\sigma_{2}^{\nu\nu}}
\{\tau(Z_{\nu}+u_{\nu\nu})+\tau \sigma_{2}^{\nu\nu}-A:(u_{\alpha\beta})\}.
\]
The positive $\tau (Z_{\nu}+u_{\nu\nu})$ absorbs the fixed constant in $\tau \sigma_{2}^{\nu\nu}-A:(u_{\alpha\beta})\geq-C(\sigma_{2}^{\nu\nu}+1)$.

For $z\ne0$, choose the first coordinate in the direction of $z$.
For $n=2$, the following inequality is an identity; for $n\geq3$, it follows by minimizing the lower-right block subject to its trace and completing the off-diagonal squares:
\[
\begin{aligned}
\mathcal K&\geq\frac{1}{n-1}\left(\epsilon|z|+(2I-dA)e\cdot\frac{z}{|z|}\right)^{2}-1+2dA(e,e)\\
&\quad+\frac{1}{2}\left|\left(I-\frac{z\otimes z}{|z|^{2}}\right)(2I-dA)e+(M-M^{T})z\right|^{2}.
\end{aligned}
\]
Using $e\cdot z=-\mu$, Young's inequality with the fixed fraction $1/(32(n-1))$, and then absorbing the error $Cd|\epsilon||z|$ in half of the first square, we obtain
\begin{equation}\label{eqn3.matrixbound}
\mathcal K\geq-1+\frac{4}{n-1}-\frac{1}{16(n-1)}
+\frac{\epsilon^{2}|z|^{2}}{2(n-1)}-\frac{4\epsilon\mu}{n-1}
-Cd-C|z|^{2}d^{2}/(\sigma_{2}^{\nu\nu})^{2}.
\end{equation}
The last loss, multiplied by $2\sigma_{2}^{\nu\nu}Q/d$, is at most $C(\Delta u)/\Lambda$.
If $z=0$, then $\mu=0$ and directly $\mathcal K\geq3-Cd$.

We reserve the fraction $1/(16(n-1))$ of $2(A_{\alpha\beta}\sigma_{2}^{\alpha\beta})Q$ for the remaining errors and use the rest to control \eqref{eqn3.matrixbound}.
If $0\leq\mu\leq1/2$ and $\epsilon\geq0$, then $|z|\geq\mu$ and
\[
\frac{\tau Z}{\sigma_{2}^{\nu\nu}}\geq\max\{1-C\eta_{0},2\epsilon\mu\}.
\]
The scalar inequality
\[
-1+\frac{4}{n-1}-\frac{1}{16(n-1)}
+\frac{(\epsilon\mu)^{2}}{2(n-1)}-\frac{4\epsilon\mu}{n-1}+\left(1-\frac{1}{16(n-1)}\right)\max\{1,2\epsilon\mu\}
\geq\frac{1}{2(n-1)}
\]
gives a uniform positive margin for sufficiently small fixed $\eta_{0}$.
If $\epsilon\mu\leq0$, the product term is nonnegative and \eqref{eqn3.curvature} gives a positive margin directly.
If both $\mu$ and $\epsilon$ are negative, the additional $2\tau Z|\mu|/\sigma_{2}^{\nu\nu}$ in \eqref{eqn3.curvature} absorbs $4\epsilon\mu/(n-1)$ once $\Lambda_{0}$ is sufficiently large.
Thus \eqref{eqn3.normal_energy} retains $c\sigma_{2}^{\nu\nu}Q/d+c(\Delta u)Q$.
Young's inequality and $Z\geq c\Lambda d$ give
\[
C\sqrt{(\Delta u)\sigma_{2}^{\nu\nu}Q/Z}\leq\frac{c}{4}\sigma_{2}^{\nu\nu}Q/d+C(\Delta u)/\Lambda.
\]
The fixed smallness of $T$ and a sufficiently large data threshold for $Q$ absorb every loss in \eqref{eqn3.normal_energy}, a contradiction.

When $\mu>1/2$, we estimate $\mathcal K$ by projection onto $e$.
By \eqref{eqn3.stationarity}, $|(u_{\alpha\nu})|^{2}\geq ZQ/4$, so $(\Delta u)\geq cQ$ and $\sigma_{2}^{\nu\nu}\leq C(1+\Lambda)$.
If $Q$ exceeds a sufficiently large fixed multiple of $1+\Lambda$, then
\[
u_{\nu\nu}\geq\frac{\Delta u}{2},\quad
-Z_{\nu}\geq\frac{\Delta u}{4},\quad \frac{1}{2}<\mu<1,\quad
\frac{d}{Z}\leq\frac{C}{\Delta u},\quad
A_{\alpha\beta}\sigma_{2}^{\alpha\beta}\geq c\Delta u.
\]
The equation and $\sigma_{2}[(u_{\alpha\beta})]\leq(n-2)(\sigma_{2}^{\nu\nu})^{2}/(2(n-1))$ imply
\[
\frac{dQ}{2\sigma_{2}^{\nu\nu}}|z|^{2}\leq1-\mu+\frac{d(Z_{\nu}+u_{\nu\nu})}{2Z}+\frac{(n-2)d\sigma_{2}^{\nu\nu}}{4(n-1)Z}.
\]
Using this inequality and \eqref{eqn3.stationarity}, the projection of $(I+M)z+(2I-dA)e$ onto $e$ is at least
\[
1-dA(e,e)-\frac{d(Z_{\nu}+u_{\nu\nu})}{2Z}
-\frac{(n-2)d\sigma_{2}^{\nu\nu}}{4(n-1)Z}-\frac{dQ}{2\sigma_{2}^{\nu\nu}}
\frac{|z||D_{T}Z+(u_{\alpha\nu})|}{|W|}.
\]
The inequality $|v|^{2}\geq2e\cdot v-1$, valid for $|e|=1$, gives
\[
\mathcal K\geq-C\Biggl(
\frac{d(Z_{\nu}+u_{\nu\nu})}{Z}
+\frac{d\sigma_{2}^{\nu\nu}}{Z}
+\frac{dQ}{\sigma_{2}^{\nu\nu}}
\frac{|z||D_{T}Z+(u_{\alpha\nu})|}{|W|}
\Biggr).
\]
After multiplication by $2\sigma_{2}^{\nu\nu}Q/d$, the losses are at most
\[
C(1+\Lambda)Q+CQ\sqrt{\Delta u}.
\]
Together with the losses already present in \eqref{eqn3.normal_energy}, they are absorbed by $2(A_{\alpha\beta}\sigma_{2}^{\alpha\beta})Q\geq c(\Delta u)Q$, by the fixed smallness of $T$ and a sufficiently large multiple of $1+\Lambda$.
The physical boundary estimate and the two interior cases prove \eqref{eqn3.normal_quotient}.
Together with \eqref{eqn3.boundary_equation}, this gives $\sup_{\partial\Omega}u_{\nu\nu}\leq C(1+\Lambda)$.

\subsection{Global second derivative estimate}

The estimates near the boundary leave a dependence on the Hessian on $\{d=T\}$.
To absorb it, we use the weak interior estimate proved in \cite[Theorem 1.2 and Section 5]{LyuD}, in the framework of \cite{KrylovWeak,ITW}, with only the numerator estimate modified below.
In the form needed here, it gives
\begin{equation}\label{eqn3.interior}
\sup_{\{d\geq\delta\}}\Delta u
\leq\varepsilon\max\left\{1,\sup_{\partial\Omega}u_{\nu\nu}\right\}
+C_{\varepsilon,\delta},\quad \varepsilon,\delta>0.
\end{equation}
The original numerator estimate uses a positive right-hand side on a neighborhood of $\overline{\Omega}$.
Semiconvexity supplies the bounded spectral tail that allows us to replace that step using only $\|f\|_{C^{1,1}(\overline{\Omega})}$.
For the quotient in \cite[Section 5]{LyuD}, whose numerator is $u_{\xi\xi}+2\eta u_{\xi}+\eta^{2}u$, consider a sufficiently large interior maximum. Then $\Delta u$ is large and stationarity in $x,\xi,\eta$ gives, for the auxiliary variables $(\xi,\eta)\in\mathbb{S}^{n}$,
\[
|D^{2}u\,\xi|\geq1,\quad
\frac{|D^{2}u_{\xi}\,\xi|}{|D^{2}u\,\xi|}\leq C_{\beta}|\xi|,
\]
where $\beta$ is the fixed denominator parameter.
Diagonalize $D^{2}u$ and set
\[
\vartheta=\frac{u_{11\xi}}{\lambda_{1}},\quad
E=D^{2}u_{\xi}-\vartheta D^{2}u.
\]
Then $E_{11}=0$, $|\vartheta|\leq C_{\beta}|\xi|+|E||\xi|$, and
\[
(\Delta u)\operatorname{tr}E
=f_{\xi}-2f\vartheta+\sum_{i=2}^{n}\lambda_{i}E_{ii}.
\]
The bounded spectral tail and bounded $f$ give, for large $\Delta u$,
\[
\begin{aligned}
-\sigma_{2}^{ij,st}u_{ij\xi}u_{st\xi}
&=|E|^{2}-(\operatorname{tr}E)^{2}-2f_{\xi}\vartheta+2f\vartheta^{2}\\
&\geq c|E|^{2}-\frac{Cf_{\xi}^{2}}{(\Delta u)^{2}}
-2f_{\xi}\vartheta+f\vartheta^{2}\\
&\geq-C_{\beta}|\xi|^{2}.
\end{aligned}
\]
The last step completes the $|E|$ square, using $|f_{\xi}|\leq C|\xi|$ and $|\xi|\leq1$.
Since
\[
\sigma_{2}^{ij}(u_{\xi\xi})_{ij}
=f_{\xi\xi}-\sigma_{2}^{ij,st}u_{ij\xi}u_{st\xi},
\]
the numerator bound in \cite[Section 5, Step 3]{LyuD} follows with the same normalization by $(n-1)\Delta u$ and requires only $\|f\|_{C^{1,1}(\overline{\Omega})}$.
The defining function and the denominator and boundary estimates in \cite[Lemma 2.7 and Section 5]{LyuD} apply unchanged.
The quotient is $C^{2}$ for $u\in C^{4}$ and a $C^{3,1}$ defining function.
The same maximum-point calculation proves \eqref{eqn3.interior} above the fixed threshold; smaller maxima are absorbed into the constant.

Equations \eqref{eqn3.reduction}, \eqref{eqn3.boundary_equation}, \eqref{eqn3.normal_quotient}, and \eqref{eqn3.Lambda} give the first inequality below.
The second follows from \eqref{eqn3.interior} with $\delta=T$:
\[
\sup_{\Omega}\Delta u
\leq C_{T}\left(1+\sup_{\{d=T\}}\sigma_{2}^{\nu\nu}\right)
\leq2C_{T}\varepsilon\sup_{\Omega}\Delta u+C_{\varepsilon,T}.
\]
Since $T$ is fixed by the data, so is $C_{T}$.
Take $\varepsilon=(4C_{T})^{-1}$ and absorb the term involving $\sup_{\Omega}\Delta u$.
Together with \eqref{eqn2.first} and $|D^{2}u|\leq\Delta u$, this proves
\begin{equation}\label{eqn3.full}
\|u\|_{C^{2}(\overline{\Omega})}\leq C.
\end{equation}

For a smooth solution with $D^{2}u\in\overline{\Gamma}_{2}$, set
\[
u_{\varepsilon}=u+\frac{\varepsilon}{2}|x|^{2},\quad
f_{\varepsilon}=f+(n-1)\varepsilon\Delta u+\binom{n}{2}\varepsilon^{2}.
\]
Then $D^{2}u_{\varepsilon}\in\Gamma_{2}$ and $f_{\varepsilon}>0$ on $\overline{\Omega}$.
For this fixed solution, choose $\varepsilon\downarrow0$ so that $(n-1)\varepsilon\|\Delta u\|_{C^{1,1}}\leq1$.
The right-hand side and boundary norms are uniformly controlled by the original data and a fixed constant.
Applying \eqref{eqn3.full} and letting $\varepsilon\downarrow0$ proves the closed-cone estimate with the same dependence.

The classical solvability theorem \cite[Theorem 2]{CNS} and the uniform estimate \eqref{eqn3.full} give a $2$-admissible solution $u\in C^{1,1}(\overline{\Omega})$ for $f\geq0$ by standard approximation.
Uniqueness follows from the comparison principle for admissible solutions.

\section{Sharpness of the boundary value assumption}\label{sec4}

The following family is a smooth strictly convex regularization of the boundary example in \cite[Example 1]{CNSDegenerate}.
The quadratic contact of the sphere at $e_{n}$ separates the regularity of the boundary trace from that of the normal second derivative.

\begin{proposition}\label{prop4.counterexample}
Let $n\geq2$ and $\alpha\in(0,1)$.
There exist smooth strictly convex functions $u_{\varepsilon}$ on $\overline{B}_{1}$, $0<\varepsilon\leq1$, solving
\[
\sigma_{2}[D^{2}u_{\varepsilon}]=f_{\varepsilon}>0\quad\text{in }B_{1},\quad
u_{\varepsilon}=\varphi_{\varepsilon}\quad\text{on }\partial B_{1},
\]
such that
\[
\|u_{\varepsilon}\|_{C^{1}(\overline{B}_{1})}
+\|f_{\varepsilon}\|_{C^{1,1}(\overline{B}_{1})}
+\|\varphi_{\varepsilon}\|_{C^{3,\alpha}(\partial B_{1})}\leq C,
\]
where $C$ depends only on $n$ and $\alpha$, while
\[
\|D^{2}u_{\varepsilon}\|_{L^{\infty}(B_{1})}\longrightarrow\infty,\quad
\inf_{B_{1}}f_{\varepsilon}\longrightarrow0\quad\text{as }\varepsilon\downarrow0.
\]
Moreover, $\sigma_{3}[D^{2}u_{\varepsilon}]\geq0$, with $\sigma_{3}\equiv0$ for $n=2$.
\end{proposition}

\begin{proof}
Set
\[
p=\frac{3+\alpha}{2}\in\left(\frac{3}{2},2\right),\quad
u_{\varepsilon}(x)=(1-x_{n}+\varepsilon)^{p}
+\frac{\varepsilon^{4-p}}{2}|x|^{2}.
\]
Its Hessian eigenvalues are $\varepsilon^{4-p}$, with multiplicity $n-1$, and
\[
\varepsilon^{4-p}+p(p-1)(1-x_{n}+\varepsilon)^{p-2}.
\]
Thus $u_{\varepsilon}$ is smooth and strictly convex, with uniformly bounded $C^{1}$ norm, and
\[
f_{\varepsilon}=(n-1)p(p-1)\varepsilon^{4-p}(1-x_{n}+\varepsilon)^{p-2}
+\binom{n}{2}\varepsilon^{8-2p}.
\]
For $j=0,1,2$, the derivatives of the first term are bounded by
\[
C\varepsilon^{4-p}\varepsilon^{p-2-j}=C\varepsilon^{2-j}.
\]
Hence $\|f_{\varepsilon}\|_{C^{1,1}(\overline{B}_{1})}\leq C$, and
$0<\inf_{B_{1}}f_{\varepsilon}\leq f_{\varepsilon}(0)\to0$.

Only a neighborhood of $e_{n}$ requires an estimate.
In the chart $x=(y,\sqrt{1-|y|^{2}})$, $|y|<1/2$, the boundary value is
\[
\varphi_{\varepsilon}(y)
=\bigl(\varepsilon+1-\sqrt{1-|y|^{2}}\bigr)^{p}
+\frac{\varepsilon^{4-p}}{2}.
\]
Here $1-\sqrt{1-|y|^{2}}$ is comparable to $|y|^{2}$; its first and third derivatives are $O(|y|)$, and its second and fourth derivatives are bounded.
Since $2p=3+\alpha$, differentiation gives
\[
|D^{j}\varphi_{\varepsilon}(y)|
\leq C(\varepsilon+|y|^{2})^{(3+\alpha-j)/2},\quad 0\leq j\leq4.
\]
For $j=0$, the constant term satisfies this bound because $4-p>p$.
The derivatives through order three are therefore uniformly bounded.

For the H\"older seminorm, take distinct $y,z$.
If
\[
|y-z|\geq\frac{1}{2}\max\{(\varepsilon+|y|^{2})^{1/2},(\varepsilon+|z|^{2})^{1/2}\},
\]
the third derivative bound gives
$|D^{3}\varphi_{\varepsilon}(y)-D^{3}\varphi_{\varepsilon}(z)|\leq C|y-z|^{\alpha}$.
Otherwise, the $1$-Lipschitz property of $(\varepsilon+|y|^{2})^{1/2}$ implies
$(\varepsilon+|w|^{2})^{1/2}\geq|y-z|$ along the segment joining $y$ to $z$.
The fourth derivative bound and $\alpha-1<0$ give the same estimate by integration along the segment.
Away from $e_{n}$, all derivatives are uniformly bounded.
Hence the boundary values have uniformly bounded $C^{3,\alpha}$ norms.

Finally,
\[
(u_{\varepsilon})_{nn}((1-\varepsilon)e_{n})
=p(p-1)(2\varepsilon)^{p-2}+\varepsilon^{4-p}\longrightarrow\infty,
\]
because $p-2=(\alpha-1)/2<0$.
The assertion about $\sigma_{3}$ follows from strict convexity.
\end{proof}

\section{Semiconvexity for the \texorpdfstring{$k$}{k}-Hessian equation on the unit ball}

We prove Theorem \ref{thm1.3} by adapting the tangential trace estimate in \cite{TrudingerBalls} to the condition $f^{1/(k-1)}\in C^{1,1}$.
All constants depend only on the quantities stated in the theorem.

\begin{proof}[Proof of Theorem \ref{thm1.3}]
Set $g=f^{1/(k-1)}$.
Since $|Df|\leq(k-1)\|Dg\|_{\infty} f^{(k-2)/(k-1)}$, the boundary comparison and the global gradient estimate in \cite[(2.1)--(2.2), Proposition 2.1]{JW} give
\[
\|u\|_{C^{1}(\overline{B}_{1})}\leq C.
\]
On the ball, these comparisons require only a $C^{2}$ extension of $\varphi$, so its $C^{3}$ norm suffices.
We first note the weighted gradient estimate
\begin{equation}\label{eqn5.gradient}
(\delta_{ab}-x_{a}x_{b})g_{a}g_{b}\leq Cg\quad\text{on }\overline{B}_{1}.
\end{equation}
Indeed, Taylor's inequality on a segment of length at most $d=1-|x|$ in the direction $-Dg$ gives $d|Dg|^{2}\leq Cg$.
For $|x|\geq1/2$ and a unit vector $\tau\perp x$, apply the one-dimensional inequality $|h'(0)|^{2}\leq2\operatorname{Lip}(h')h(0)$ to the nonnegative function
\[
h(t)=g(x\cos t+|x|\tau\sin t).
\]
This bounds the tangential gradient by $|D_{T}g|^{2}\leq Cg$.
Combining the two bounds with $1-|x|^{2}\leq2d$ proves \eqref{eqn5.gradient}; for $|x|\leq1/2$, the segment estimate alone suffices.
Only values of $g$ in the closed ball are used.

Consider the tangential trace and its quadratic correction
\[
v=\Delta u-x_{a}x_{b}u_{ab},\quad w=v+|x|^{2}.
\]
All second differentiations below hold almost everywhere, since $g$ and $D^{2}u$ belong to $C^{1,1}$.
Twice differentiating the equation and using homogeneity gives
\begin{equation}\label{eqn5.traceidentity}
\sigma_{k}^{ij}v_{ij}
=(\delta_{ab}-x_{a}x_{b})
\bigl(f_{ab}-\sigma_{k}^{ij,st}u_{ija}u_{stb}\bigr)
-4x_{a}f_{a}-2kf.
\end{equation}
Concavity of $\sigma_{k}^{1/k}$ and \eqref{eqn5.gradient} imply
\[
\begin{aligned}
\sigma_{k}^{ij}v_{ij}
&\geq(k-1)g^{k-2}(\delta_{ab}-x_{a}x_{b})g_{ab}
-\frac{k-1}{k}g^{k-3}(\delta_{ab}-x_{a}x_{b})g_{a}g_{b}\\
&\quad-4(k-1)g^{k-2}x\cdot Dg-2kg^{k-1}
\geq-Cg^{k-2}.
\end{aligned}
\]
Newton's inequalities give
\begin{equation}\label{eqn5.5}
\sum_{i=1}^{n}\sigma_{k}^{ii}
=(n-k+1)\sigma_{k-1}[D^{2}u]
\geq c(\Delta u)^{1/(k-1)}g^{k-2},\quad |D^{2}u|\leq\Delta u.
\end{equation}
Consequently,
\begin{equation}\label{eqn5.normalized}
\frac{\sigma_{k}^{ij}w_{ij}}{\sum_{i=1}^{n}\sigma_{k}^{ii}}
\geq2-C(\Delta u)^{-1/(k-1)}\quad\text{in }B_{1}.
\end{equation}

On the boundary,
\[
v=\Delta_{\partial B_{1}}\varphi+(n-1)Du\cdot x,
\]
so $w\leq C$ there.
Suppose that the maximum of $w$ exceeds a sufficiently large fixed constant and is attained at $x_{0}\in B_{1}$.
Since $v\leq2\Delta u$, the Laplacian is large at $x_{0}$ and remains large in a neighborhood.
For each fixed solution, the normalized coefficients in \eqref{eqn5.normalized} are bounded and uniformly positive definite on a sufficiently small closed ball about $x_{0}$.
When $\Delta u$ is sufficiently large, \eqref{eqn5.normalized} is strictly positive almost everywhere.
Since $w\in C^{1,1}$, the strong maximum principle \cite[Chapter 9]{GT} excludes an interior maximum.
The ellipticity constants enter only this qualitative comparison.
Thus $v\leq w\leq C$ on $\overline{B}_{1}$.

For $x\ne0$, the restriction of $D^{2}u$ to $x^{\perp}$ belongs to $\overline{\Gamma}_{k-1}\subset\overline{\Gamma}_{2}$.
Its trace is therefore nonnegative and bounds its norm.
Since
\[
v=(1-|x|^{2})\Delta u
+|x|^{2}\operatorname{tr}(D^{2}u|_{x^{\perp}}),
\]
we have $\Delta u\leq C$ for $|x|\leq1/2$ and $\operatorname{tr}(D^{2}u|_{x^{\perp}})\leq C$ for $|x|\geq1/2$.
The min--max principle gives $\lambda_{2}(D^{2}u)\leq C$.
Finally, admissibility implies $\Delta u-\lambda_{1}(D^{2}u)\geq0$, whence
\[
\lambda_{n}(D^{2}u)\geq-\sum_{i=2}^{n-1}\lambda_{i}(D^{2}u)\geq-(n-2)C.
\]
This proves the theorem.
\end{proof}

The bound passes to locally uniform limits through the second-difference inequality on line segments in $B_{1}$.

\section*{Acknowledgements}

The author is deeply grateful to Professors Dongsheng Li and Yu Yuan for their guidance and support.
This work grew out of about three years of sustained research by the author on the open problem posed by Ivochkina-Trudinger-Wang \cite{ITW}.
The author's initial insight into the semiconvexity structure was inspired by the Legendre-Lewy-Yuan rotation used in interior second derivative estimates for the $2$-Hessian equation.
AI tools were used as an aid in preliminary mathematical exploration and in refining the exposition.
All mathematical arguments were independently verified by the author, who takes full responsibility for the manuscript.

\end{document}